%% file: main.tex
\documentclass[12pt,a4paper]{amsart}

\usepackage{amsmath,amssymb,amsthm,amsfonts,graphicx,color}
\usepackage{epstopdf}
\usepackage{color}
\usepackage{comment}
\usepackage[colorlinks,  urlcolor=blue]{hyperref}   
\usepackage{bbm} 

\newtheorem{theorem}{Theorem}[section]
\newtheorem{lemma}[theorem]{Lemma}

\newtheorem{proposition}[theorem]{Proposition}
\newtheorem{definition}[theorem]{Definition}
\newtheorem{example}{Example}
\newtheorem{conjecture}[theorem]{Conjecture}

\newtheorem{remark}[theorem]{Remark}

\DeclareSymbolFont{AMSb}{U}{msb}{m}{n}
\DeclareMathSymbol{\N}{\mathbin}{AMSb}{"4E}
\DeclareMathSymbol{\Z}{\mathbin}{AMSb}{"5A}
\DeclareMathSymbol{\R}{\mathbin}{AMSb}{"52}
\DeclareMathSymbol{\Q}{\mathbin}{AMSb}{"51}
\DeclareMathSymbol{\I}{\mathbin}{AMSb}{"49}
\DeclareMathSymbol{\C}{\mathbin}{AMSb}{"43}

\title[Inference for Fractional Diffusions]{Statistical Inference for Fractional Diffusions}
\author[P. R. Alonso-Martin, H. Boedihardjo, and A. Papavasiliou]
{Pablo Ramses Alonso-Martin, Horatio Boedihardjo and Anastasia Papavasiliou}

\begin{document}



\maketitle

\input{introduction3}
\input{inference3}

\input{multiscale3}

\end{document}

%% file: introduction3.tex
\section{Introduction to Fractional Diffusions}
\label{sec:introduction}

\subsection{Introduction}

In continuous-time stochastic modelling, one quickly encounters equations of the form
\[
\mathrm{d}X_t=\mu\left(X_t\right)\mathrm{d}t+\sigma\left(X_t\right)\mathrm{d}Z_t,
\]
where $Z$ is a stochastic process. In the multidimensional setting, one considers
\begin{equation}
\mathrm{d}X_t^{i}=\mu^{i}\left(X_t^{1},\ldots,X_t^{d}\right)\mathrm{d}t+\sum_{j=1}^{n}\sigma_{j}^{i}\left(X_t^{1},\ldots,X_t^{d}\right)\mathrm{d}Z_t^{j}.
\label{eq:MultidimensionalSDE-1}
\end{equation}

Classically, the driving signal $Z_t$ is taken to be Brownian motion or, more generally, a martingale. Since the 1980s, and possibly earlier, there has been considerable interest in extending stochastic calculus to the case where the driving signal is a fractional Brownian motion (see, for example, \cite{Mishura08,Nualart,BookfBm}). A one-dimensional fractional Brownian motion with Hurst parameter $H$ is a continuous real-valued Gaussian process $\left(B_t^{H}\right)_{t\geq0}$ with covariance function
\begin{equation}
\mathbb{E}\left[B_t^{H}B_s^{H}\right]=\frac{1}{2}\left(t^{2H}+s^{2H}-\left|t-s\right|^{2H}\right).
\label{eq:fBmCovariance}
\end{equation}

We start by a short review of fundamental concepts and results from rough path theory, that allow us to properly define solutions to \eqref{eq:MultidimensionalSDE-1}. In Section 2, we review the literature on fitting \eqref{eq:MultidimensionalSDE-1} to data. Finally, in Section 3, we discuss the inference problem in the case where \eqref{eq:MultidimensionalSDE-1} arises as a result of homogenisation.

\subsection{Differential equations driven by fractional Brownian motion}

Since $B_t^{H}$ does not, in general, have finite variation (except when $H=1$), and is not a martingale (except when $H=\frac{1}{2}$), we start by explaining what equation \eqref{eq:MultidimensionalSDE-1} means when the driving signals $Z^{i}$ are independent fractional Brownian motions.

If $B^{H,1},\ldots,B^{H,d}$ are independent one-dimensional fractional Brownian motions, we write
\[
B^H=(B^{H,1},\ldots,B^{H,d})
\]
for the corresponding $d$-dimensional fractional Brownian motion. From this point onward, $B^H$ denotes the vector-valued process, while $B^{H,i}$ denotes its $i$th component.

For each $m\in\mathbb{N}$, define the dyadic piecewise linear interpolation
\[
B_t^{H,i}(m)=B_{\frac{k}{2^{m}}}^{H,i}+2^{m}\left(t-\frac{k}{2^{m}}\right)\left(B_{\frac{k+1}{2^{m}}}^{H,i}-B_{\frac{k}{2^{m}}}^{H,i}\right),\qquad
t\in\left[\frac{k}{2^{m}},\frac{k+1}{2^{m}}\right],
\]
and set
\[
B_t^{H}(m)=\left(B^{H,1}(m),\ldots,B^{H,d}(m)\right).
\]
As $m\to\infty$, $B_t^H(m)\to B_t^H$ for each $t$. Since the paths $t\mapsto B_t^H(m)$ are smooth away from the dyadic points
\[
\left\{\frac{k}{2^m}:k\in\{0,1,\ldots,2^m\}\right\},
\]
the ordinary differential equation (ODE)
\begin{equation}
\frac{\mathrm{d}X_t^{i}(m)}{\mathrm{d}t}
=
\mu^{i}\left(X_t^{1}(m),\ldots,X_t^{d}(m)\right)
+
\sum_{j=1}^{n}\sigma_{j}^{i}\left(X_t^{1}(m),\ldots,X_t^{d}(m)\right)\frac{\mathrm{d}B_t^{H,j}(m)}{\mathrm{d}t},
\label{eq:FirstInterval}
\end{equation}
with initial condition
\[
X_0^{i}(m)=x_0^{i}\in\mathbb{R},\qquad i=1,\dots,d,
\]
is well defined on $\left[0,2^{-m}\right]$, provided that the coefficients $\mu^{i}$ and $\sigma_{j}^{i}$ are sufficiently smooth; for example, it is enough to assume that these functions, together with all of their derivatives, are bounded.

Having solved the equation on $\left[0,2^{-m}\right]$, one obtains the values $X_{2^{-m}}^{i}(m)$, which are then used as the initial condition on the next interval $\left[2^{-m},2\cdot 2^{-m}\right]$. Proceeding iteratively in this way yields a solution $X^{i}(m)$ on $[0,1]$. Rough path theory guarantees that the limit
\[
X_t^{i}=\lim_{m\to\infty}X^{i}(m)
\]
exists and solves the differential equation
\[
\mathrm{d}X_t^{i}
=
\mu^{i}\left(X_t^{1},\ldots,X_t^{d}\right)
+
\sum_{j=1}^{n}\sigma_{j}^{i}\left(X_t^{1},\ldots,X_t^{d}\right)\mathrm{d}B_t^{H,j}
\]
in a sense that will be described in the following sections.

\subsection{$H>\frac{1}{2}$ case: Young integration}

In general, the more irregular the sample paths of $t\mapsto Z_t$ are, the more difficult it is to make sense of \eqref{eq:MultidimensionalSDE-1}. For the purposes of understanding \eqref{eq:MultidimensionalSDE-1}, a natural way to measure the regularity of a sample path is through its Hölder exponent.

\begin{definition}
(a) A function $z:[0,1]\to\mathbb{R}^d$ is said to be $h$-Hölder continuous if there exists a constant $c\geq 0$ such that, for all $0\leq s<t\leq 1$,
\[
|z_t-z_s|\leq c|t-s|^h.
\]

(b) A stochastic process $Z$ is said to have $h$-Hölder continuous sample paths if, almost surely, the function $t\mapsto Z_t(\omega)$ is $h$-Hölder continuous.
\end{definition}

The larger $h$ is, the more regular the sample paths of $Z$ are. Sample paths of Brownian motion, for example, are $(\frac{1}{2}-\varepsilon)$-Hölder continuous for all $\varepsilon>0$. A consequence of \eqref{eq:fBmCovariance} is that
\[
\mathbb{E}\big[(B_t^H-B_s^H)^2\big]=|t-s|^{2H},
\]
which implies that, almost surely, the sample paths are $(H-\varepsilon)$-Hölder continuous for every $\varepsilon>0$. When $H>\frac{1}{2}$, the sample paths of $B_t^H$ are therefore strictly more regular than those of Brownian motion. In this case, the differential equation \eqref{eq:MultidimensionalSDE-1} can be interpreted in the sense of Young \cite{Young36}. In particular, one has the following result.

\begin{lemma}[{\cite{Young36}}]
\label{lemma:YoungIntegration}
If $y:[0,1]\to\mathbb{R}$ and $z:[0,1]\to\mathbb{R}$ are $\alpha$-Hölder continuous and $\beta$-Hölder continuous, respectively, with $\alpha+\beta>1$, then the limit
\[
\int_s^t y_u\,dz_u
=
\lim_{|\mathcal P|\to 0}
\sum_{i=0}^{n-1} y_{t_i}\bigl(z_{t_{i+1}}-z_{t_i}\bigr)
\]
exists. Here, $\lim_{|\mathcal P|\to 0}$ denotes the limit as the mesh size of partitions $\mathcal P=(t_0,\ldots,t_n)$ of $[s,t]$ tends to zero.
\end{lemma}

We will often consider compositions of a ``sufficiently smooth'' function with Hölder continuous paths. The following definition gives the appropriate notion of smoothness. Here $\lfloor\gamma\rfloor$ denotes the floor of $\gamma$, that is, the largest integer less than or equal to $\gamma$.

\begin{definition}
Let $\gamma\geq 1$. A function $F:\mathbb{R}^d\to\mathbb{R}$ is said to be $\mathrm{Lip}(\gamma)$ if the following hold.

\emph{Case 1: $\gamma\notin\mathbb{N}$.}
The function $F$ is bounded and has $\lfloor\gamma\rfloor$ bounded derivatives. Its $\lfloor\gamma\rfloor$th derivative is $(\gamma-\lfloor\gamma\rfloor)$-Hölder continuous.

\emph{Case 2: $\gamma\in\mathbb{N}$.}
The function $F$ is bounded and has $\gamma-1$ bounded derivatives. Its $(\gamma-1)$st derivative is Lipschitz continuous.

More generally, if $V$ is a finite-dimensional normed vector space, we say that $F:\mathbb{R}^d\to V$ is $\mathrm{Lip}(\gamma)$ if, with respect to some basis of $V$, each coordinate component of $F$ is $\mathrm{Lip}(\gamma)$.
\end{definition}

If $F$ is $\mathrm{Lip}(\gamma)$ for some $\gamma>\frac{1}{H-\varepsilon}-1$ and $y$ is $(H-\varepsilon)$-Hölder continuous, then $F(y)$ is $\gamma(H-\varepsilon)$-Hölder continuous. In particular, when $H>\frac{1}{2}$, there exists $\varepsilon>0$ such that
\[
\gamma(H-\varepsilon)+H-\varepsilon>1.
\]
Consequently, for each $(H-\varepsilon)$-Hölder continuous function $t\mapsto z_t$, the map
\[
I:y\mapsto
\left(
y_0^j+\sum_i\int_0^\cdot F_i^j(y_u)\,dz_u^i
\right)_{j=1}^d
\]
is a well-defined map from the space of $(H-\varepsilon)$-Hölder continuous functions into itself.

A special case of \cite{Lyons94} shows that the equation
\begin{equation}
y_t^j
=
y_0^j+\sum_i\int_0^t F_i^j(y_u)\,dz_u^i,
\label{eq:YoungRDE}
\end{equation}
with the integral understood in the sense of Young, has a solution if $\gamma>\frac{1}{H-\varepsilon}-1$, and this solution is unique if $\gamma>\frac{1}{H-\varepsilon}$.

In fact, one obtains not only existence and uniqueness of solutions, but also continuity of the solution map with respect to the driving signal $z$. Within rough path theory, two topologies are commonly used, corresponding to the following two metrics. For a path $z:[0,1]\to\mathbb{R}^d$, the first metric is the $\alpha$-Hölder metric,
\[
d_{\alpha}^{(1)}(z,\tilde z)
=
\sup_{s<t}
\frac{|z_t-z_s-(\tilde z_t-\tilde z_s)|}{|t-s|^\alpha}.
\]
The second is the $p$-variation metric,
\[
d_{p\text{-}\mathrm{var}}^{(1)}(z,\tilde z)
=
\|z-\tilde z\|_{p\text{-}\mathrm{var}},
\]
with $\|z\|_{p\text{-}\mathrm{var}}$ being the $p$-variation norm defined by
\[
\|z\|_{p\text{-}\mathrm{var}}
=
\left(
\sup_{\mathcal P}
\sum_{i=0}^{n-1}
|z_{t_{i+1}}-z_{t_i}|^p
\right)^{1/p},
\]
where the supremum is taken over all finite partitions $\mathcal P$ of $[0,1]$.

The map sending the driving path $z$ to the solution $y$ of \eqref{eq:YoungRDE} is continuous with respect to both $d_{\alpha}^{(1)}$ for $\alpha<H-\varepsilon$ and $d_{p\text{-}\mathrm{var}}^{(1)}$ for $p>\frac{1}{H-\varepsilon}$. The $\alpha$-Hölder metric is generally easier to define and work with, while the $p$-variation topology is often better suited to the analysis of differential equations such as \eqref{eq:MultidimensionalSDE-1}. Certain results, such as the Cass--Litterer--Lyons estimate \cite{CassLittererLyons}, arise more naturally from the $p$-variation point of view.

\subsection{$\frac{1}{4}<H\leq\frac{1}{2}$ case: rough path regime}

Let $Z=(Z^1,\ldots,Z^d)$ be a $d$-dimensional stochastic process. A central idea in rough path theory is that if the (as yet undefined) iterated integrals
\[
\left(
\int_{s<t_1<\cdots<t_n<t}
dZ_{t_1}^{i_1}\cdots dZ_{t_n}^{i_n}
\right)_{i_1,\ldots,i_n}
\]
associated with a sample path $Z(\omega)$ are well defined and satisfy the appropriate algebraic and analytic properties, then the solutions to equations of the form \eqref{eq:MultidimensionalSDE-1} can be constructed deterministically from $Z$. In the next subsection, we explain how these iterated integrals should be understood.

\subsubsection{Iterated integrals of fractional Brownian motion}
\label{subsec:IteratedIntegralfBm}

Let $H\in\left(\frac{1}{4},\frac{1}{2}\right]$. In this case, the Hölder regularity of fractional Brownian motion sample paths is too low for iterated integrals of the form
\[
\left(
\int_{s<t_1<t_2<t} dB_{t_1}^{H,i_1}\,dB_{t_2}^{H,i_2}
\right)_{i_1,i_2=1}^d
\quad\text{or}\quad
\left(
\int_{s<t_1<t_2<t_3<t} dB_{t_1}^{H,i_1}\,dB_{t_2}^{H,i_2}\,dB_{t_3}^{H,i_3}
\right)_{i_1,i_2,i_3=1}^d
\]
to be defined in the sense of Young integration from Lemma \ref{lemma:YoungIntegration}. A first approach to overcoming this difficulty is due to Coutin and Qian \cite{CoutinQian}, building on ideas from earlier work such as \cite{WongZakai}, and we now briefly outline it.

Recall that $B_t^H(m)$ denotes the dyadic piecewise linear interpolation of fractional Brownian motion. As $m\to\infty$, $B_t^H(m)$ converges to $B_t^H$ for each $t$, so $B^H(m)$ is a piecewise smooth approximation of $B^H$. Since $B^H(m)$ is piecewise smooth, the second- and third-order iterated integrals
\begin{align*}
B_{s,t}^{H,i_1,i_2}(m)
&:=
\int_s^t \int_s^{t_2}
dB_{t_1}^{H,i_1}(m)\,dB_{t_2}^{H,i_2}(m),\\
B_{s,t}^{H,i_1,i_2,i_3}(m)
&:=
\int_s^t \int_s^{t_3}\int_s^{t_2}
dB_{t_1}^{H,i_1}(m)\,dB_{t_2}^{H,i_2}(m)\,dB_{t_3}^{H,i_3}(m)
\end{align*}
are well defined. These quantities should be viewed as functions on
\[
\Delta_{0,1}:=\{(s,t):0\leq s\leq t\leq 1\}.
\]

It is therefore natural to define the iterated integrals as the limits
\begin{align*}
B_{s,t}^{H,i_1,i_2}
&:=
\lim_{m\to\infty}
\int_s^t \int_s^{t_2}
dB_{t_1}^{H,i_1}(m)\,dB_{t_2}^{H,i_2}(m),\\
B_{s,t}^{H,i_1,i_2,i_3}
&:=
\lim_{m\to\infty}
\int_s^t \int_s^{t_3}\int_s^{t_2}
dB_{t_1}^{H,i_1}(m)\,dB_{t_2}^{H,i_2}(m)\,dB_{t_3}^{H,i_3}(m).
\end{align*}

According to \cite{Lyons02,CoutinQian}, not only do these limits exist, but the convergence of $B^{H,i_1,i_2}(m)$ also takes place in the $p$-variation metric. More precisely, for $Z,\tilde Z:\Delta_{0,1}\to\mathbb{R}$, define
\[
d_{p\text{-}\mathrm{var}}^{(2)}(Z,\tilde Z)
:=
\sup_{\mathcal P}
\left(
\sum_{i=0}^{n-1}
|Z_{t_i,t_{i+1}}-\tilde Z_{t_i,t_{i+1}}|^{p/2}
\right)^{2/p},
\]
where the supremum is taken over all partitions $\mathcal P=(t_0,\ldots,t_n)$ of $[0,1]$. Similarly, define
\[
d_{p\text{-}\mathrm{var}}^{(3)}(Z,\tilde Z)
:=
\sup_{\mathcal P}
\left(
\sum_{i=0}^{n-1}
|Z_{t_i,t_{i+1}}-\tilde Z_{t_i,t_{i+1}}|^{p/3}
\right)^{3/p}.
\]
Then, for $p>\frac{1}{H}$, as $m\to\infty$, almost surely
\begin{align*}
d_{p\text{-}\mathrm{var}}^{(2)}
\bigl(B^{H,i_1,i_2}(m),B^{H,i_1,i_2}\bigr)
&\to 0,\\
d_{p\text{-}\mathrm{var}}^{(3)}
\bigl(B^{H,i_1,i_2,i_3}(m),B^{H,i_1,i_2,i_3}\bigr)
&\to 0.
\end{align*}

Similar convergence also holds with respect to the $\alpha$-Hölder metrics \cite{Friz10}. More precisely, define
\begin{align*}
d_{\alpha}^{(2)}(Z,\tilde Z)
&:=
\sup_{s<t}
\frac{|Z_{s,t}-\tilde Z_{s,t}|^{1/2}}{|t-s|^\alpha},\\
d_{\alpha}^{(3)}(Z,\tilde Z)
&:=
\sup_{s<t}
\frac{|Z_{s,t}-\tilde Z_{s,t}|^{1/3}}{|t-s|^\alpha}.
\end{align*}
Then, for any $\alpha<H$, the approximations also converge in $d_{\alpha}^{(2)}$ and $d_{\alpha}^{(3)}$.

\subsubsection{Tensor products}

From an algebraic point of view, it is more convenient to embed iterated integrals of the form
\[
\left(
\int_s^t \cdots \int_s^{t_2}
dB_{t_1}^{H,i_1}(m)\cdots dB_{t_n}^{H,i_n}(m)
\right)_{i_1,\ldots,i_n}
\]
into a truncated tensor algebra. We refer the reader to standard texts on tensor products, for example \cite{RyanTensorProducts}, for a general introduction. For our purposes, it is enough to think of the tensor product as a bilinear, generally non-commutative, way of multiplying vectors together.

Let $(e_1,\ldots,e_d)$ be the standard basis of $\mathbb{R}^d$. By convention,
\[
(\mathbb{R}^d)^{\otimes 0}=\mathbb{R}.
\]
The tensor product $(\mathbb{R}^d)^{\otimes n}$ is a $d^n$-dimensional vector space with basis
\[
\left(
e_{i_1}\otimes e_{i_2}\otimes\cdots\otimes e_{i_n}
\right)_{i_1,\ldots,i_n=1}^d.
\]

The truncated tensor algebra of degree $n$ is defined by
\[
T^{(n)}(\mathbb{R}^d)
=
\left\{
(a_0,a_1,\ldots,a_n): a_i\in(\mathbb{R}^d)^{\otimes i}
\right\}.
\]
It is equipped with the operations of addition,
\[
(a_0,a_1,\ldots,a_n)+(b_0,b_1,\ldots,b_n)
=
(a_0+b_0,a_1+b_1,\ldots,a_n+b_n),
\]
and multiplication,
\begin{equation}
(a_0,a_1,\ldots,a_n)\otimes(b_0,b_1,\ldots,b_n)
=
(c_0,c_1,\ldots,c_n),
\label{eq:TensorProduct-1-1}
\end{equation}
where, for each $k=0,\ldots,n$,
\[
c_k=\sum_{i=0}^k a_i\otimes b_{k-i}.
\]

The product operation $\otimes$ is generally non-commutative, in the sense that $a\otimes b\neq b\otimes a$, but it is associative,
\[
(a\otimes b)\otimes c = a\otimes(b\otimes c),
\]
and distributive:
\begin{align*}
(\lambda_1 a_1+\mu_1 b_1)\otimes(\lambda_2 a_2+\mu_2 b_2)
&=
\lambda_1\lambda_2\,a_1\otimes a_2
+\lambda_1\mu_2\,a_1\otimes b_2 \\
&\quad
+\mu_1\lambda_2\,b_1\otimes a_2
+\mu_1\mu_2\,b_1\otimes b_2.
\end{align*}
Moreover,
\[
1\otimes b=b
\qquad\text{for all } b\in T^{(n)}(\mathbb{R}^d).
\]

\subsubsection{Properties of iterated integrals for fractional Brownian motion}

We now discuss algebraic and analytic properties of the iterated integrals of deterministic paths, which will later also apply to fractional Brownian motion sample paths. These properties will enable us to construct solutions to \eqref{eq:MultidimensionalSDE-1} from the iterated integrals.

Let $Z:[0,1]\to\mathbb{R}^d$ be a piecewise smooth function. (Later we will specialise to the case $Z=B^H(m)$, the piecewise linear interpolation of fractional Brownian motion.) With the tensor-product notation introduced above, we may embed second-order iterated integrals of the form
\[
\left(
\int_s^t \int_s^{t_2} dZ_{t_1}^{i_1}\,dZ_{t_2}^{i_2}
\right)_{i_1,i_2}
\]
into $(\mathbb{R}^d)^{\otimes 2}$ as
\[
\int_s^t \int_s^{t_2} dZ_{t_1}\otimes dZ_{t_2}
=
\sum_{i_1,i_2}
\left(
\int_s^t \int_s^{t_2} dZ_{t_1}^{i_1}\,dZ_{t_2}^{i_2}
\right)
e_{i_1}\otimes e_{i_2},
\]
and a similar identity holds for third-order iterated integrals.

For $x\in(\mathbb{R}^d)^{\otimes n}$ with expansion
\[
x=\sum_{i_1,\ldots,i_n} x^{i_1\cdots i_n}\,
e_{i_1}\otimes\cdots\otimes e_{i_n},
\]
we define
\[
\|x\|
=
\left(
\sum_{i_1,\ldots,i_n}(x^{i_1\cdots i_n})^2
\right)^{1/2}.
\]

This embedding not only allows us to express iterated integrals in a coordinate-free way, but also provides a convenient algebraic framework for describing their properties. For instance, for $s\leq u\leq t$, see for example \cite{LyonsCaruanaLevy},
\begin{equation}
\int_s^t \int_s^{t_2} dZ_{t_1}\otimes dZ_{t_2}
=
\int_s^u \int_s^{t_2} dZ_{t_1}\otimes dZ_{t_2}
+
\int_s^u dZ_{t_1}\otimes \int_u^t dZ_{t_2}
+
\int_u^t \int_u^{t_2} dZ_{t_1}\otimes dZ_{t_2},
\label{eq:SecondOrderChen}
\end{equation}
and
\begin{align}
\int_s^t \int_s^{t_3}\int_s^{t_2}
dZ_{t_1}\otimes dZ_{t_2}\otimes dZ_{t_3}
&=
\int_s^u \int_s^{t_3}\int_s^{t_2}
dZ_{t_1}\otimes dZ_{t_2}\otimes dZ_{t_3}
\nonumber\\
&\quad+
\int_s^u \int_s^{t_2}
dZ_{t_1}\otimes dZ_{t_2}\otimes \int_u^t dZ_{t_3}
\nonumber\\
&\quad+
\int_s^u dZ_{t_1}\otimes
\int_u^t \int_u^{t_2}
dZ_{t_2}\otimes dZ_{t_3}
\nonumber\\
&\quad+
\int_u^t \int_u^{t_3}\int_u^{t_2}
dZ_{t_1}\otimes dZ_{t_2}\otimes dZ_{t_3}.
\label{eq:ThirdOrderChen}
\end{align}

A tidier way to express these identities is to collect the iterated integrals together and define
\begin{align*}
S^{(2)}(Z)_{s,t}
&=
\left(
1,
\int_s^t dZ_{t_1},
\int_s^t \int_s^{t_2} dZ_{t_1}\otimes dZ_{t_2}
\right)
\in T^{(2)}(\mathbb{R}^d),\\
S^{(3)}(Z)_{s,t}
&=
\left(
1,
\int_s^t dZ_{t_1},
\int_s^t \int_s^{t_2} dZ_{t_1}\otimes dZ_{t_2},
\int_s^t \int_s^{t_3}\int_s^{t_2}
dZ_{t_1}\otimes dZ_{t_2}\otimes dZ_{t_3}
\right)
\in T^{(3)}(\mathbb{R}^d).
\end{align*}

The infinite series
\[
S(Z)_{s,t}
=
\left(
1,
\int_s^t dZ_{t_1},
\int_s^t \int_s^{t_2} dZ_{t_1}\otimes dZ_{t_2},
\ldots
\right)
\]
is called the \emph{path signature} of $Z$, while $S^{(2)}(Z)$ and $S^{(3)}(Z)$ are the truncated signatures of $Z$.

Then algebraic identities such as \eqref{eq:SecondOrderChen} and \eqref{eq:ThirdOrderChen} translate into
\begin{align*}
S^{(2)}(Z)_{s,t}
&=
S^{(2)}(Z)_{s,u}\otimes S^{(2)}(Z)_{u,t},\\
S^{(3)}(Z)_{s,t}
&=
S^{(3)}(Z)_{s,u}\otimes S^{(3)}(Z)_{u,t}.
\end{align*}

Another important property of iterated integrals is the shuffle identity:
\begin{align}
&\int_{s<t_1<\cdots<t_n<t}
dZ_{t_1}^{i_1}\cdots dZ_{t_n}^{i_n}
\int_{s<t_{n+1}<\cdots<t_{n+k}<t}
dZ_{t_{n+1}}^{i_{n+1}}\cdots dZ_{t_{n+k}}^{i_{n+k}}
\nonumber\\
&\qquad=
\sum_{\sigma\in Sh(n,k)}
\int_{s<t_1<\cdots<t_{n+k}<t}
dZ_{t_1}^{i_{\sigma^{-1}(1)}}\cdots
dZ_{t_{n+k}}^{i_{\sigma^{-1}(n+k)}},
\label{eq:Shuffle}
\end{align}
where $Sh(n,k)$ denotes the set of all permutations $\sigma$ such that
\[
\sigma^{-1}(1)<\cdots<\sigma^{-1}(n),
\qquad
\sigma^{-1}(n+1)<\cdots<\sigma^{-1}(n+k).
\]

We now take $Z$ to be a sample path of the piecewise linear interpolation $B^H(m)$ of fractional Brownian motion. Then \eqref{eq:SecondOrderChen}, \eqref{eq:ThirdOrderChen}, and \eqref{eq:Shuffle} imply that
\begin{align*}
S^{(2)}\bigl(B^H(m)\bigr)_{s,t}
&=
S^{(2)}\bigl(B^H(m)\bigr)_{s,u}\otimes
S^{(2)}\bigl(B^H(m)\bigr)_{u,t},\\
S^{(3)}\bigl(B^H(m)\bigr)_{s,t}
&=
S^{(3)}\bigl(B^H(m)\bigr)_{s,u}\otimes
S^{(3)}\bigl(B^H(m)\bigr)_{u,t},
\end{align*}
and
\[
B_{s,t}^{H,i_1\cdots i_n}(m)\,
B_{s,t}^{H,i_{n+1}\cdots i_{n+k}}(m)
=
\sum_{\sigma\in Sh(n,k)}
B_{s,t}^{H,i_{\sigma^{-1}(1)}\cdots i_{\sigma^{-1}(n+k)}}(m).
\]

By taking limits as $m\to\infty$, we see that analogous identities hold for the iterated integrals of fractional Brownian motion constructed in Section \ref{subsec:IteratedIntegralfBm}. More precisely, define
\begin{align*}
\mathbb{B}_{s,t}^{H,1}
&=
\sum_{i_1} B_{s,t}^{H,i_1} e_{i_1},\\
\mathbb{B}_{s,t}^{H,2}
&=
\sum_{i_1,i_2}
B_{s,t}^{H,i_1,i_2}\, e_{i_1}\otimes e_{i_2},\\
\mathbb{B}_{s,t}^{H,3}
&=
\sum_{i_1,i_2,i_3}
B_{s,t}^{H,i_1,i_2,i_3}\,
e_{i_1}\otimes e_{i_2}\otimes e_{i_3},
\end{align*}
and let
\begin{align*}
S^{(2)}(B^H)_{s,t}
&=
\left(1,\mathbb{B}_{s,t}^{H,1},\mathbb{B}_{s,t}^{H,2}\right),\\
S^{(3)}(B^H)_{s,t}
&=
\left(1,\mathbb{B}_{s,t}^{H,1},\mathbb{B}_{s,t}^{H,2},
\mathbb{B}_{s,t}^{H,3}\right).
\end{align*}
Then, for $s\leq u\leq t$ and $n+k\leq 3$,
\begin{align*}
S^{(2)}(B^H)_{s,t}
&=
S^{(2)}(B^H)_{s,u}\otimes S^{(2)}(B^H)_{u,t},\\
S^{(3)}(B^H)_{s,t}
&=
S^{(3)}(B^H)_{s,u}\otimes S^{(3)}(B^H)_{u,t},\\
B_{s,t}^{H,i_1\cdots i_n}\,
B_{s,t}^{H,i_{n+1}\cdots i_{n+k}}
&=
\sum_{\sigma\in Sh(n,k)}
B_{s,t}^{H,i_{\sigma^{-1}(1)}\cdots i_{\sigma^{-1}(n+k)}}.
\end{align*}

In addition, because of convergence in Hölder topology, the iterated integrals of fractional Brownian motion are also Hölder continuous. In particular, if $H\in\left(\frac{1}{3},\frac{1}{2}\right]$ and $\alpha\in\left(\frac{1}{3},H\right)$, then the map
\[
(s,t)\mapsto S^{(2)}(B^H)_{s,t}
\]
is an $\alpha$-Hölder geometric rough path. Similarly, if $H\in\left(\frac{1}{4},\frac{1}{3}\right]$ and $\alpha\in\left(\frac{1}{4},H\right)$, then
\[
(s,t)\mapsto S^{(3)}(B^H)_{s,t}
\]
is an $\alpha$-Hölder geometric rough path in the following sense.

\begin{definition}
\label{definition:geometricrp}
Let $\alpha\in(0,1]$ and let $N=\lfloor 1/\alpha\rfloor$. Let
\[
\mathbb{Z}:\Delta_{0,1}\to T^{(N)}(\mathbb{R}^d)
\]
be a function. Let $Z_{s,t}^{i_1,\ldots,i_j}$ denote the coefficient of $e_{i_1}\otimes\cdots\otimes e_{i_j}$ in the expansion of $\mathbb{Z}_{s,t}$ with respect to the basis
\[
\left\{
e_{i_1}\otimes\cdots\otimes e_{i_j}
\right\}_{i_1,\ldots,i_j}.
\]
We say that $\mathbb{Z}$ is an $\alpha$-Hölder geometric rough path on $\mathbb{R}^d$ if the following properties hold.

(i) For all $s\leq u\leq t$,
\[
\mathbb{Z}_{s,u}\otimes \mathbb{Z}_{u,t}=\mathbb{Z}_{s,t}.
\]

(ii) For all $n,k$ with $n+k\leq N$,
\[
Z_{s,t}^{i_1\cdots i_n}\,
Z_{s,t}^{i_{n+1}\cdots i_{n+k}}
=
\sum_{\sigma\in Sh(n,k)}
Z_{s,t}^{i_{\sigma^{-1}(1)}\cdots i_{\sigma^{-1}(n+k)}}.
\]

(iii) For every $n\leq N$ and every choice of indices $i_1,\ldots,i_n$,
\[
\sup_{s<t}
\frac{|Z_{s,t}^{i_1\cdots i_n}|}{(t-s)^{n\alpha}}<\infty.
\]
\end{definition}

A similar definition may be given in the $p$-variation setting. Given a function
\[
(s,t)\mapsto \mathbb{Z}_{s,t}
=
\bigl(1,\mathbb{Z}_{s,t}^1,\ldots,\mathbb{Z}_{s,t}^N\bigr)
\in T^{(N)}(\mathbb{R}^d),
\]
its $p$-variation norm is defined by
\[
\|\mathbb{Z}\|_{p\text{-}\mathrm{var}}
=
\max_{1\leq m\leq N}
\sup_{\mathcal P}
\left(
\sum_{i=0}^{n-1}
\|\mathbb{Z}_{t_i,t_{i+1}}^m\|^{p/m}
\right)^{m/p},
\]
where the supremum is taken over all partitions $\mathcal P=(t_0,\ldots,t_n)$ of $[0,1]$. The corresponding $p$-variation metric is defined by
\[
d_{p\text{-}\mathrm{var}}(\mathbb{Z},\tilde{\mathbb{Z}})
=
\|\mathbb{Z}-\tilde{\mathbb{Z}}\|_{p\text{-}\mathrm{var}}.
\]

A $p$-geometric rough path is a function
\[
\mathbb{Z}:\Delta_{0,1}\to T^{(N)}(\mathbb{R}^d)
\]
satisfying properties (i) and (ii) of Definition \ref{definition:geometricrp}, with $\|\mathbb{Z}\|_{p\text{-}\mathrm{var}}<\infty$.

\begin{remark}
We use the same notation $\|\cdot\|_{p\text{-}\mathrm{var}}$ for both the $p$-variation norm of an ordinary path and the $p$-variation norm of a rough path. The meaning is determined by the object under consideration: for a path $z:[0,1]\to\mathbb{R}^d$, $\|z\|_{p\text{-}\mathrm{var}}$ denotes the usual path $p$-variation norm, whereas for a rough path $\mathbb{Z}=(1,\mathbb{Z}^1,\ldots,\mathbb{Z}^N)$, $\|\mathbb{Z}\|_{p\text{-}\mathrm{var}}$ denotes the rough-path $p$-variation norm. Accordingly, $d_{p\text{-}\mathrm{var}}^{(1)}$ refers to the path-level metric, while $d_{p\text{-}\mathrm{var}}$ refers to the $p$-rough path metric.
\end{remark}

\subsubsection{Integrals against $\alpha$-Hölder geometric rough paths}

Let $\mathbb{Z}=(1,\mathbb{Z}^{1},\mathbb{Z}^{2})$ be an $\alpha$-Hölder geometric rough path on $\mathbb{R}^{d}$, and let $Z:[0,1]\to\mathbb{R}^{d}$ be a function such that, for all $s<t$,
\[
\mathbb{Z}_{s,t}^{1}=Z_t-Z_s.
\]
Let $f:\mathbb{R}^{d}\to\mathbb{R}$ be a $\mathrm{Lip}(\gamma)$ function with $\gamma>\frac{1}{\alpha}-1$. We first discuss how to define integrals of the form
\[
\int_s^t f(Z_u)\,dZ_u.
\]

When $\alpha<\frac{1}{2}$, Young's notion of integration from Lemma \ref{lemma:YoungIntegration} no longer yields a finite limit, and a different definition is needed.

To motivate the definition, note that if $u$ and $s$ are close, then
\[
f(Z_u)
=
f(Z_s)+Df(Z_s)(Z_u-Z_s)+O(|u-s|^{2\alpha}),
\]
where, for $Z_u=(Z_u^1,\ldots,Z_u^d)$,
\[
Df(Z_s)(Z_u-Z_s)
=
\sum_{i=1}^d \frac{\partial f}{\partial x_i}(Z_s)(Z_u^i-Z_s^i).
\]
We therefore expect that, when $s$ and $t$ are close,
\begin{align*}
\int_s^t f(Z_u)\,dZ_u
&=
\int_s^t f(Z_s)\,dZ_u
+\int_s^t Df(Z_s)(Z_u-Z_s)\,dZ_u
+O(|t-s|^{3\alpha})\\
&=
f(Z_s)\int_s^t dZ_u
+Df(Z_s)\int_s^t \int_s^u dZ_v\otimes dZ_u
+O(|t-s|^{3\alpha}),
\end{align*}
where the linear map $Df(Z_s)$ is defined by
\[
Df(Z_s)(e_i\otimes e_j)
=
\frac{\partial f}{\partial x_i}(Z_s)e_j.
\]

Recall the intuition that if $\mathbb{Z}_{s,t}=(1,\mathbb{Z}_{s,t}^{1},\mathbb{Z}_{s,t}^{2})$ is an $\alpha$-Hölder geometric rough path with $\alpha\in\left(\frac{1}{3},\frac{1}{2}\right]$, then $\mathbb{Z}_{s,t}^{1}$ represents $\int_s^t dZ_u$, while $\mathbb{Z}_{s,t}^{2}$ intuitively represents $\int_s^t\int_s^u dZ_v\otimes dZ_u$. Therefore, when $s$ and $t$ are close,
\[
\int_s^t f(Z_u)\,dZ_u
=
f(Z_s)\mathbb{Z}_{s,t}^{1}
+
Df(Z_s)\mathbb{Z}_{s,t}^{2}
+
O(|t-s|^{3\alpha}).
\]

To turn this approximation into an exact definition of the integral, we subdivide a long interval into small pieces and sum the corresponding approximations. Let $\mathcal P$ be a partition of the interval $[s,t]$. We say that a function $g$ of partitions converges to $L\in\mathbb{R}$ as $|\mathcal P|\to 0$ if, for every $\varepsilon>0$, there exists $\delta>0$ such that, for every partition $\mathcal P$ satisfying $|\mathcal P|<\delta$, we have
\[
|g(\mathcal P)-L|<\varepsilon.
\]
In this case, we write
\[
\lim_{|\mathcal P|\to 0} g(\mathcal P)=L.
\]
Then, for any partition of $[s,t]$,
\begin{align*}
\int_s^t f(Z_u)\,dZ_u
&=
\sum_{i=0}^{n-1}\int_{t_i}^{t_{i+1}} f(Z_u)\,dZ_u\\
&=
\sum_{i=0}^{n-1}
\left(
f(Z_{t_i})\mathbb{Z}_{t_i,t_{i+1}}^{1}
+
Df(Z_{t_i})\mathbb{Z}_{t_i,t_{i+1}}^{2}
+
O(|t_{i+1}-t_i|^{3\alpha})
\right).
\end{align*}
Note that
\begin{align*}
\sum_{i=0}^{n-1} O(|t_{i+1}-t_i|^{3\alpha})
&\leq
O(|\mathcal P|^{3\alpha-1})
\sum_{i=0}^{n-1}|t_{i+1}-t_i|\\
&=
O(|\mathcal P|^{3\alpha-1})(t-s)\to 0
\end{align*}
as $|\mathcal P|\to 0$, since $3\alpha>1$. Therefore, at least formally,
\[
\int_s^t f(Z_u)\,dZ_u
=
\lim_{|\mathcal P|\to 0}
\sum_{i=0}^{n-1}
\left(
f(Z_{t_i})\mathbb{Z}_{t_i,t_{i+1}}^{1}
+
Df(Z_{t_i})\mathbb{Z}_{t_i,t_{i+1}}^{2}
\right).
\]
This motivates the following definition.

\begin{lemma}[{\cite{LyonsCaruanaLevy}}]
\label{lemma:IntegralOneForm}
Let $\mathbb{Z}=(1,\mathbb{Z}^{1},\mathbb{Z}^{2})\in T^{(2)}(\mathbb{R}^d)$ be an $\alpha$-Hölder geometric rough path with $\alpha\in\left(\frac{1}{3},\frac{1}{2}\right]$. Let $f$ be a $\mathrm{Lip}(\gamma)$ function with $\gamma>\frac{1}{\alpha}-1$. Then the limit
\[
\int_s^t f(Z_u)\,dZ_u
:=
\lim_{|\mathcal P|\to 0}
\sum_{i=0}^{n-1}
\left(
f(Z_{t_i})\mathbb{Z}_{t_i,t_{i+1}}^{1}
+
Df(Z_{t_i})\mathbb{Z}_{t_i,t_{i+1}}^{2}
\right)
\]
exists.

Moreover, let $\mathbb{Z}(m)$ be a sequence of $\alpha$-Hölder geometric rough paths, and let $Z_{s,t}^{i_1,\ldots,i_j}(m)$ denote the coefficient of $e_{i_1}\otimes\cdots\otimes e_{i_j}$ in the expansion of $\mathbb{Z}_{s,t}(m)$ with respect to the basis
\[
\{e_{i_1}\otimes\cdots\otimes e_{i_j}\}_{i_1,\ldots,i_j}.
\]
If $Z^{i}(m)$ converges to $Z^{i}$ for all $i$ in $d_{\alpha}^{(1)}$, and $Z^{i_1,i_2}(m)$ converges to $Z^{i_1,i_2}$ in $d_{\alpha}^{(2)}$ as $m\to\infty$, then
\[
\lim_{m\to\infty}
\int_s^t f(Z_u(m))\,dZ_u(m)
=
\int_s^t f(Z_u)\,dZ_u.
\]
\end{lemma}

A similar definition and limiting result hold for $\alpha\in\left(\frac{1}{4},\frac{1}{3}\right]$, provided the definition of the integral is modified to
\[
\int_s^t f(Z_u)\,dZ_u
:=
\lim_{|\mathcal P|\to 0}
\sum_{i=0}^{n-1}
\left(
f(Z_{t_i})\mathbb{Z}_{t_i,t_{i+1}}^{1}
+
Df(Z_{t_i})\mathbb{Z}_{t_i,t_{i+1}}^{2}
+
D^2f(Z_{t_i})\mathbb{Z}_{t_i,t_{i+1}}^{3}
\right),
\]
where $D^2f$ is the second derivative operator of $f$, defined by
\[
D^2f(x)(e_i\otimes e_j\otimes e_k)
=
\frac{\partial^2 f}{\partial x_i\partial x_j}(x)e_k.
\]

\subsubsection{Integrals of controlled rough paths}

Recall that when $\alpha\in\left(\frac{1}{2},1\right]$ and $z$ is $\alpha$-Hölder continuous, one can define integrals of the form
\begin{equation}
\int_0^1 y_u\,dz_u.
\label{eq:Integral}
\end{equation}
under fairly general assumptions on $y$. It is therefore natural to ask what condition on $y$ ensures that the integral \eqref{eq:Integral} is well defined when $\alpha\in\left(\frac{1}{3},\frac{1}{2}\right]$, beyond the example $y_u=f(z_u)$ considered in the previous subsection.

Looking at the proof of Lemma \ref{lemma:IntegralOneForm}, the main property of $Df$ used there is the bound
\[
\left|f(z_t)-f(z_s)-Df(z_s)(z_t-z_s)\right|
\leq c|z_t-z_s|^{\gamma-1}.
\]
A controlled rough path satisfies an analogous estimate, which allows the proof of Lemma \ref{lemma:IntegralOneForm} to be extended more generally.

For this definition, we write $L(V,W)$ for the space of linear maps from a vector space $V$ to a vector space $W$. Since we work only with finite-dimensional spaces, the particular choice of norms is not important, but for definiteness we fix the following conventions:
\begin{enumerate}
\item $\|\cdot\|_{\mathbb{R}^n}$ denotes the Euclidean norm on $\mathbb{R}^n$;
\item for $A\in L\bigl((\mathbb{R}^d)^{\otimes(i+1)},\mathbb{R}^n\bigr)$,
\[
\|A\|=\sup_{\|x\|=1}\|A(x)\|_{\mathbb{R}^n}.
\]
\end{enumerate}
We use the following definition.

\begin{definition}[{\cite{GubControlRP}}]
Let $\alpha\in\left(\frac{1}{N+1},\frac{1}{N}\right]$. Given an $\alpha$-Hölder geometric rough path $\mathbb{Z}=(1,\mathbb{Z}^1,\ldots,\mathbb{Z}^N)$ on $\mathbb{R}^d$, a path controlled by $\mathbb{Z}$ is a collection
\[
\mathbb{Y}=\bigl(Y^{(0)},Y^{(1)},\ldots,Y^{(N-1)}\bigr)
\]
with $N$ components, where $Y^{(0)}$ maps $[0,1]$ into $\mathbb{R}^k$, each $Y^{(i)}$ maps $[0,1]$ into $L\bigl((\mathbb{R}^d)^{\otimes i}\otimes\mathbb{R}^k,\mathbb{R}^n\bigr)$, and there exists $M\in\mathbb{R}$ such that, for all $(s,t)\in\Delta_{0,1}$,
\begin{align*}
\|Y_t^{(i)}\| &\leq M,\\
\left\|
Y_t^{(i)}-\sum_{j=i}^{N-1}Y_s^{(j)}\mathbb{Z}_{s,t}^{\,j-i}
\right\|
&\leq M|t-s|^{\alpha(N+1-i)}.
\end{align*}
\end{definition}

When $N=2$ and $\mathbb{Z}_{s,t}^1=Z_t-Z_s$, the map
\[
u\mapsto \bigl(f(Z_u),Df(Z_u)\bigr)
\]
is a controlled rough path for any bounded function $f:\mathbb{R}^d\to\mathbb{R}$ with bounded derivatives of all orders. In particular, one has the following result.

\begin{lemma}[{\cite{GubControlRP}}]
\label{lemma:IntegralOneForm-1}
Let $\alpha\in\left(\frac{1}{N+1},\frac{1}{N}\right]$. Let $\mathbb{Z}=(1,\mathbb{Z}^1,\ldots,\mathbb{Z}^N)\in T^{(N)}(\mathbb{R}^d)$ be an $\alpha$-Hölder geometric rough path, and let
\[
\mathbb{Y}=\bigl(Y^{(0)},\ldots,Y^{(N-1)}\bigr)
\]
be a path controlled by $\mathbb{Z}$. Then:

\begin{enumerate}
\item The limit
\[
\int_s^t \mathbb{Y}_u\,dZ_u
:=
\lim_{|\mathcal P|\to 0}
\sum_{i=0}^{n-1}\sum_{j=0}^{N-1}
Y_{t_i}^{(j)}\mathbb{Z}_{t_i,t_{i+1}}^{\,j+1}
\]
exists, and the function
\[
t\mapsto
\left(
\int_s^t \mathbb{Y}_u\,dZ_u,\,
Y_t^{(0)},\ldots,Y_t^{(N-2)}
\right)
\]
is again a path controlled by $\mathbb{Z}$.

\item Let $N=2$. If $f:\mathbb{R}^d\to L(\mathbb{R}^n,\mathbb{R}^k)$ is a $\mathrm{Lip}(\gamma)$ function with $\gamma\in(N+1,N+2]$, then
\[
t\mapsto
\bigl(f(Y_t^{(0)}),Df(Y_t^{(0)})\bigr)
\]
is a path controlled by $\mathbb{Z}$.

\item Let $N=3$. If $f:\mathbb{R}^d\to L(\mathbb{R}^n,\mathbb{R}^k)$ is a $\mathrm{Lip}(\gamma)$ function with $\gamma\in(N+1,N+2]$, then
\[
t\mapsto
\bigl(
f(Y_t^{(0)}),
Df(Y_t^{(0)}),
D^2f(Y_t^{(0)})Y_t^{(1)}
\bigr)
\]
is a path controlled by $\mathbb{Z}$.
\end{enumerate}
\end{lemma}

\subsubsection{Solutions to differential equations driven by fractional Brownian motion}

With the notion of controlled paths, we are now in a position to specify what equation \eqref{eq:MultidimensionalSDE-1} means when $\mathbb{Z}$ is an $\alpha$-Hölder geometric rough path.

\begin{theorem}[{\cite{GubControlRP}}]
\label{theorem:ExistenceUniquenessRDE}
Let $N\in\mathbb{N}$. Let $\mathbb{Z}$ be an $\alpha$-Hölder geometric rough path on $\mathbb{R}^d$ with $\alpha\in\left(\frac{1}{N+1},\frac{1}{N}\right]$. Let $f:\mathbb{R}^n\to L(\mathbb{R}^d,\mathbb{R}^n)$ be a $\mathrm{Lip}(\gamma)$ map, with $\gamma\in(N+1,N+2]$.

(i) If $N=2$, then there exists a unique path
\[
\mathbb{Y}=\bigl(Y^{(0)},Y^{(1)}\bigr)
\]
controlled by
\[
(s,t)\mapsto \mathbb{Z}_{s,t}
=
\bigl(1,Z_t-Z_s,\mathbb{Z}_{s,t}^2\bigr)
\]
such that
\begin{align*}
Y_t^{(0)}
&=
Y_0^{(0)}+\int_0^t f\bigl(Y_u^{(0)}\bigr)\,dZ_u,\\
Y_t^{(1)}
&=
f\bigl(Y_t^{(0)}\bigr).
\end{align*}

(ii) If $N=3$, then there exists a unique path
\[
\mathbb{Y}=\bigl(Y^{(0)},Y^{(1)},Y^{(2)}\bigr)
\]
controlled by
\[
(s,t)\mapsto \mathbb{Z}_{s,t}
=
\bigl(1,Z_t-Z_s,\mathbb{Z}_{s,t}^2,\mathbb{Z}_{s,t}^3\bigr)
\]
such that
\begin{align*}
Y_t^{(0)}
&=
Y_0^{(0)}+\int_0^t f\bigl(Y_u^{(0)}\bigr)\,dZ_u,\\
Y_t^{(1)}
&=
f\bigl(Y_t^{(0)}\bigr),\\
Y_t^{(2)}
&=
Df\bigl(Y_t^{(0)}\bigr)Y_t^{(1)}.
\end{align*}
\end{theorem}

In particular, if $H\in\left(\frac{1}{3},\frac{1}{2}\right]$, then taking
\[
\mathbb{Z}_{s,t}=S^{(2)}(B^H)_{s,t}
\]
gives a solution to a differential equation driven by fractional Brownian motion $B^H$. Similarly, if $H\in\left(\frac{1}{4},\frac{1}{3}\right]$ and one takes
\[
\mathbb{Z}_{s,t}=S^{(3)}(B^H)_{s,t},
\]
then the same conclusion holds.

The method outlined in this chapter, however, does not extend to the case $H\leq \frac{1}{4}$, since in that regime one cannot show that
\[
S^{(2)}\bigl(B^H(m)\bigr)
\]
converges as $m\to\infty$. For the purposes of this chapter, when $\mathbb{Z}$ is taken to be $S^{(2)}(B^H)_{s,t}$ or $S^{(3)}(B^H)_{s,t}$, we call the function $Y^{(0)}$ in Theorem \ref{theorem:ExistenceUniquenessRDE} a \emph{fractional diffusion}.

%% file: inference3.tex
\section{Statistical Inference for Fractional Diffusions}
\label{sec:inference}

\subsection{Literature review}

Consider the model
\begin{equation}
\label{eq: inf main}
\mathrm{d}X_t = f_\theta(X_t)\mathrm{d}t + \sigma_\theta(X_t)\mathrm{d}B_t^H,
\end{equation}
where $\theta\in \Theta \subseteq \mathbb{R}^d$ and $B_t^H$ is fractional Brownian motion with Hurst parameter $H\in(0,1)$. In the case where $\sigma_\theta \equiv \sigma$, the integrals appearing in the solution are against time and can be interpreted using Young's construction. In the case of multiplicative noise and for $H\leq \frac{1}{2}$, one must specify the precise interpretation of \eqref{eq: inf main}. Following section \ref{sec:introduction}, we assume that $H>\frac{1}{4}$ and interpret \eqref{eq: inf main} as the limit of solutions to the dyadic piecewise linear interpolations of $B^H$.

In this section, we investigate how to fit these models to data. When fitting a model to data, one must develop statistical methods for estimating the parameter $\theta$ from some type of observation of the solution. Statistical inference has been extensively studied in the case $H=\frac{1}{2}$ (stochastic differential equations, or SDEs), but the general case $H\in(0,1)$ is much less developed, with the first results appearing in the early 2000s. We begin with a brief review of the current state of the literature.

First, it is worth noting that statistical inference can take many different forms, depending on the context:
\begin{itemize}
\item[(i)] What data do we expect to obtain? Does this correspond to one trajectory of the solution, $\{X(\omega)_t; t\in [0,T]\}$, or many trajectories, $\{X(\omega_j)_t; t\in [0,T], j=1,\dots,n\}$? If many trajectories are available, are they independent or interacting?
\item[(ii)] How are these trajectories sampled? Data are always finite objects, and the most common type of data consists of discrete-time samples of the trajectory (or trajectories). A typical assumption is that one observes a time series
\begin{equation}
\label{eq: observations}
x_{{\mathcal D}_{\delta,T}}:= \{X(\omega)_{t_i}; t_i\in{\mathcal D}_{\delta,T}\},
\qquad
\text{where }
{\mathcal D}_{\delta,T}=\{k\delta : k=0,\dots,N=T/\delta\}.
\end{equation}
for some fixed $\delta>0$ and $T<+\infty$. It is possible, however, to observe different features of the trajectory, such as frequencies or path signatures.
\item[(iii)] What are the sampling limitations? Does one control the sampling rate, so that $\delta$ may be assumed arbitrarily small? Can one assume an arbitrarily large time horizon $T=N\delta$? Is it possible to choose both $T$ and $\delta$ for the data set? These limitations determine the type of limit that should be considered in the asymptotic analysis, namely $\delta\to 0$, $T\to \infty$, or both. If both limits are taken simultaneously, what are the optimal rates? For example, if $\delta = \varepsilon^\alpha$ and $T = \varepsilon^{-\beta}$, what are the optimal values of $\alpha,\beta>0$?
\item[(iv)] What model refinements should be considered? Is $H$ assumed known? Is the noise additive or multiplicative? Are different parameters allowed for the drift $f_\theta$ and diffusion coefficient $\sigma_\theta$? Is either $f_\theta$ or $\sigma_\theta$ assumed known? Can one assume a specific family of functions for $f_\theta$ and $\sigma_\theta$?
\item[(v)] What properties are required of the estimators? For example, what balance is needed between precision, robustness, and computational cost?
\end{itemize}

Most of the literature focuses on the case where one is given discrete-time samples from a single trajectory. While generalising to many discretely observed independent trajectories is straightforward, this setting becomes more interesting when both $\delta>0$ and $T<\infty$ are fixed but the number of trajectories $n$ may be taken arbitrarily large. Examples include \cite{Papavasiliou11, Semnani25}, where the authors assume independent observations of specific coordinates of the path signature of $X$ over the whole time horizon $[0,T]$, rather than discrete-time samples. Another interesting extension involving multiple trajectories concerns statistical inference for interacting particle systems (see \cite{Amorino23}), which has so far only been studied in the standard diffusion framework $H=\frac{1}{2}$. In the rest of this chapter, we follow the standard assumption that the data take the form \eqref{eq: observations}, while noting that the other directions mentioned above also contain many interesting open problems.

One of the first papers to study statistical inference for $H\neq \frac{1}{2}$ is \cite{Kleptsyna02}. That paper considers the fractional Ornstein--Uhlenbeck process (fOU),
\begin{equation}
\label{eq: fOU}
\mathrm{d}X_t = -\theta X_t \mathrm{d}t + \sigma \mathrm{d}B_t^H.
\end{equation}
Assuming that $H$ is known, with $H>\frac{1}{2}$, and that a trajectory is observed continuously, the authors construct maximum likelihood estimators for both $\sigma^2$ and $\theta$ and prove their consistency, with asymptotic normality established some years later in \cite{Brouste10,Bercu11}. Notably, the two parameters are estimated sequentially: first $\sigma^2$ is estimated, and then, assuming $\sigma$ known, $\theta$ is estimated.

This line of work was subsequently extended in two main directions. The first concerns an in-depth study of statistical inference under a variety of assumptions within the fOU setting \cite{Kleptsyna02,Hu10,Hu19,Brouste13,Haress21}. This was then extended to the slightly more general fractional Vasicek model \cite{Lohvinenko17,Lohvinenko19},
\begin{equation}
\label{eq: Vasicek}
\mathrm{d}X_t = \kappa(\mu-X_t)\mathrm{d}t + \sigma \mathrm{d}B_t^H,
\end{equation}
for which there has been considerable recent activity in the econometrics literature \cite{Xiao19a,Xiao19b,Shi24,Wang25}, particularly because fractional Brownian motion appears better suited to many financial models \cite{Gatheral18}. A further extension is given in \cite{Nourdin19}, where the authors consider general Hermite processes in place of $B_t^H$. The second direction concerns broader generalisations to more complex models \cite{Tudor07,Chronopoulou13,Lysy13,Neuenkirch14,Panloup20,Haress24}. Each of these directions is discussed in more detail below.

\subsubsection{Fractional Ornstein--Uhlenbeck model}

We first discuss the literature on the fOU model from the perspective of questions (iii)--(v) above. One of the first papers to consider statistical inference for the fOU model is \cite{Kleptsyna02}, where the authors construct the maximum likelihood estimator (MLE) $\hat{\theta}^{MLE}_T$ for $\theta$, assuming that $\sigma$ and $H$ are known, that $H>\frac{1}{2}$, and that the trajectory is observed continuously. They prove strong consistency as $T\to\infty$.

Still under the continuous-observation assumption, \cite{Hu10} proposes a least-squares-type estimator (LSE),
\[
\hat{\theta}_T^{LSE} = \frac{\int_0^T X_t \,\mathrm{d}X_t}{\int_0^T X_t^2 \,\mathrm{d}t},
\]
and a moment-matching estimator (MME), obtained from the ergodic limit $\frac{1}{T}\int_0^T X_t^2\,\mathrm{d}t$:
\[
\hat{\theta}^{MME}_T = \left( \frac{1}{\sigma^2 H \Gamma(2H)T}\int_0^T X_t^2 \,\mathrm{d}t\right)^{-\frac{1}{2H}}.
\]
The authors prove strong consistency as $T\to\infty$, assuming that $\sigma$ and $H$ are known and that $H\geq \frac{1}{2}$, and they establish asymptotic normality for $H\in(\frac{1}{2},\frac{3}{4})$.

The MME in \cite{Hu10} forms the basis for estimators constructed from discrete observations. In \cite{Brouste13}, the authors construct an estimator for $\theta$ based on discrete observations $x_{{\mathcal D}_{\delta,T}}$ by discretising $\hat{\theta}^{MME}$, and they prove its consistency and asymptotic normality when $\delta\to 0$ and $T\to\infty$ simultaneously at an appropriate rate, for $H\in(\frac{1}{2},\frac{3}{4})$. Later, \cite{Hu19} extends the results of \cite{Hu10} to all $H\in(0,1)$, including the proof of consistency and asymptotic normality of the discretised $\hat{\theta}^{MME}$ under simultaneous limits $\delta\to 0$ and $T\to\infty$ at an appropriate rate.

In \cite{Brouste13}, the authors also construct estimators for $(\sigma,H)$ from discrete data and prove their consistency and asymptotic normality for all $H\in(0,1)$. However, when estimating $\theta$, the parameters $(\sigma,H)$ are assumed to be known. Separating the estimation of $\theta$ from that of $(\sigma,H)$ is standard in the literature, since the information content for $\sigma$ and $H$ lies mainly at fine temporal scales, whereas the information for $\theta$ lies in long-time behaviour. Many papers focus solely on the drift parameter, while others estimate all parameters but in a separated manner. In contrast, \cite{Haress21} considers simultaneous estimation of the three parameters $(\sigma,H,\theta)$. The authors construct an estimator for this three-dimensional vector using ergodic limits over two different time increments, that is, sums of the form
\[
\frac{1}{N}\sum_{k=1}^N g(X_{k\delta}, X_{(k+1)\delta}),
\]
and prove consistency and asymptotic normality for fixed $\delta>0$ and $T\to\infty$, under the assumption that the true parameters lie in an appropriate subset of the three-dimensional parameter space.

This evolution is quite intuitive. Regarding question (iii), a natural starting point is to assume continuous observation, or sequential limits of the form $\lim_{T\to\infty}\lim_{\delta\to 0}$. The next step is to consider discrete observations with a sampling rate tending to zero simultaneously with $T\to\infty$, that is, limits of the form $\delta\to 0$ and $N\delta^p\to\infty$ for appropriate $p>0$, or equivalently $\varepsilon\to 0$ with $\delta=\varepsilon^\alpha$ and $T=\varepsilon^{-\beta}$ for suitable $\alpha,\beta>0$. The final step is to treat $\delta>0$ as fixed and construct an estimator that converges as $N\to\infty$. This corresponds to constructing an exact estimator for a discretely observed trajectory, which is challenging even for $H=\frac{1}{2}$; see \cite{Craigmile23} for a comprehensive review of exact estimators for diffusion parameters and an extensive list of references.

Regarding question (iv), the natural progression is first to treat the drift and diffusion parameters, and possibly $H$, separately, and then to consider joint estimation, for the reasons outlined above. With respect to question (v), there are three main techniques for constructing estimators: likelihood-based, least-squares, and moment-matching approaches, all of which typically use ergodic limits. While MLEs are expected to be the most efficient, in the sense of minimising asymptotic variance, and are therefore a natural starting point, the practical difficulty of computing them has motivated the development of alternative estimators whose asymptotic properties are easier to analyse and which may be more computationally efficient.

\subsubsection{Fractional Vasicek model}

Following the discussion of the fOU case, we now turn to the literature on statistical inference for the fractional Vasicek model. This model is of particular interest in econometrics, as it is commonly used to model interest rates under the ``rough market'' assumption \cite{Gatheral18}. We note that the literature also considers the alternative parametrisation of \eqref{eq: Vasicek} given by
\begin{equation}
\label{eq: Vasicek2}
\mathrm{d}X_t = (\alpha - \beta X_t)\mathrm{d}t + \sigma \mathrm{d}B_t^H.
\end{equation}
Since the fOU model is a special case of the Vasicek model, corresponding to $\mu=0$ or $\alpha=0$, the focus is on understanding the effect of the additional parameter on the inference problem. The literature on fractional Vasicek models also considers the non-ergodic case, corresponding to $\kappa<0$ or $\beta<0$.

In \cite{Lohvinenko17}, the authors construct the MLE for $(\alpha,\beta)$ in \eqref{eq: Vasicek2}, assuming continuous observations, $\sigma=1$, known $H>\frac{1}{2}$, and $\beta>0$ (the ergodic case). They consider three situations: estimating $\alpha$ with $\beta$ known, estimating $\beta$ with $\alpha$ known, and estimating both simultaneously, proving consistency and asymptotic normality in each case. These results are extended in \cite{Lohvinenko19} to the non-ergodic case $\beta<0$.

In \cite{Xiao19a}, the results of \cite{Hu10} are extended to the fractional Vasicek model \eqref{eq: Vasicek}. The authors construct estimators for both $\kappa$ and $\mu$ under the assumptions of known $\sigma$, known $H>\frac{1}{2}$, and continuous observations. They prove consistency and asymptotic normality, treating each estimator separately rather than jointly. Both ergodic and non-ergodic cases are considered, corresponding to $\kappa>0$, $\kappa<0$, and $\kappa=0$. In \cite{Xiao19b}, these results are extended to the regime $H<\frac{1}{2}$.

In \cite{Shi24}, the authors move on to consider discrete observations of \eqref{eq: Vasicek} for $H\in(0,1)$ and $\kappa>0$. They use the spectral density of the discrete-time series to construct approximate Whittle maximum likelihood (AWML) estimators for the four parameters $(\kappa,\sigma,H,\mu)$, treating the drift parameters and diffusion parameters separately. As in \cite{Haress21}, they prove consistency and asymptotic normality for fixed $\delta>0$ and $T\to\infty$.

In \cite{Wang25}, the authors construct the exact MLE for all four parameters $(\kappa,\sigma,H,\mu)$ simultaneously under the ergodicity assumption, using discrete observations. They prove consistency for all estimators and asymptotic normality for the estimators of $(\kappa,\sigma,H)$ and $\mu$ separately. They also demonstrate the optimality of their estimators through numerical simulations. This represents the natural goal of constructing exact joint estimators for all parameters when $\delta>0$ is fixed. The remaining challenge is computational efficiency. For completeness, we also mention the recent paper \cite{Bennedsen26}, where the authors propose more flexible models for rough stochastic volatility and discuss computationally efficient inference methods based on composite likelihoods.

\subsubsection{General models with additive noise}

The other main direction in which the literature has expanded concerns more general models than the fOU and fractional Vasicek models, which are restricted to linear drift and constant diffusion coefficient. Significant progress has been made in the study of general models with additive noise, that is, models of the form
\begin{equation}
\label{eq: add noise}
\mathrm{d}X_t = f_\theta(X_t)\mathrm{d}t + \sigma \mathrm{d}B_t^H.
\end{equation}
The first paper to consider a nonlinear drift is \cite{Tudor07}, where the authors examine the case $f_\theta(x) = \theta f(x)$ with $H\in(0,1)$. Building on \cite{Kleptsyna02}, they construct the MLE for $\theta$ given continuous observations and prove consistency. By discretising their MLE, they also construct an estimator for discrete observations and prove consistency when $T\to\infty$ and $\delta\to 0$ at an appropriate joint rate.

In \cite{Neuenkirch14}, the authors consider a fairly general class of drift functions $f_\theta$, assuming that $\sigma$ and $H>\frac{1}{2}$ are known. They construct an estimator for $\theta$ from discrete observations by minimising the score function defined as the difference between the theoretical and observed quadratic variations of the noise:
\begin{equation}
\label{eq: Neuenkirch score}
Q_{N,\delta}(\theta) := \frac{1}{N\delta^2}\sum_{k=1}^N \left( |\Delta X_{k\delta} - f_\theta(X_{k\delta})\delta|^2 - \sigma^2 \delta^{2H}\right),
\end{equation}
where $\Delta X_{k\delta}= X_{(k+1)\delta} - X_{k\delta}$. This construction relies on applying the Euler scheme to discretise \eqref{eq: add noise} and on the convergence of that scheme for $H>\frac{1}{2}$ and $\delta\to0$. Thus, although the estimator is constructed directly from discrete data, rather than by discretising an estimator derived for continuous observations, it is not exact for fixed $\delta>0$. Consistency requires that $T=N\delta\to\infty$ and $\delta\to0$ jointly.

In \cite{Panloup20}, the authors extend the results of \cite{Neuenkirch14} by considering a wider class of drift functions $f_\theta$ and allowing any $H\in(0,1)$, still assuming that $H$ is known and that $\sigma_\theta$ is constant and known. Their construction is based on the convergence of the empirical measure $\frac{1}{N}\sum_{k=1}^N \delta_{X_{k\delta}}$ to a probability measure $\nu_{\theta_0}$ characterised by the parameter $\theta_0$. If $\nu_\theta$ were known explicitly, one could define the exact estimator
\begin{equation}
\label{eq: Panloup estimator 1}
\hat{\theta}^e_{N} = \arg\min_\theta d\left(\frac{1}{N}\sum_{k=1}^N \delta_{X_{k\delta}},\nu_\theta \right),
\end{equation}
where $d$ is an appropriate distance on the space of measures, for example the Wasserstein distance. Since $\nu_\theta$ is not known in general, the authors propose approximating it by simulating solutions of \eqref{eq: inf main} via an Euler scheme. If $\{Z^{\gamma,\theta}_{k\gamma}; k = 1,\dots,M\}$ is the Euler approximation with step size $\gamma$ and time horizon $\gamma M$, the approximate estimator is defined as
\begin{equation}
\label{eq: Panloup estimator 2}
\hat{\theta}^a_{\gamma,N,M} = \arg\min_\theta d\left(\frac{1}{N}\sum_{k=1}^N \delta_{X_{k\delta}}, \frac{1}{M}\sum_{k=1}^M \delta_{Z^{\gamma,\theta}_{k\gamma}}\right).
\end{equation}
For the exact estimator $\hat{\theta}^e_N$, consistency can be shown as $N\to\infty$, that is, as $T\to\infty$, with fixed $\delta>0$. The approximate estimator $\hat{\theta}^a_{\gamma,N,M}$ is also consistent for fixed $\delta>0$, but requires the sequential limit $\lim_{\gamma\to 0}\lim_{N,M\to\infty}$. Thus, although the estimator is exact with respect to $\delta$, reducing the approximation error incurs significant computational cost. In this sense, the paper constructs ``exact'' estimators for the drift parameters of discretely observed trajectories for any $H\in(0,1)$ as $T\to\infty$.

Next, \cite{Haress24} combines ideas from \cite{Panloup20} and \cite{Haress21} to construct an ``exact'' estimator for discrete observations that simultaneously estimates the drift parameter $\theta$ in $f_\theta$, the diffusion coefficient $\sigma$, and the Hurst parameter $H$. As in \cite{Panloup20}, the estimation is based on minimising a distance similar to \eqref{eq: Panloup estimator 2}, but on an augmented state space incorporating several time lags of the process. The results in \cite{Haress24} come close to the ultimate goal of constructing an exact joint estimator for all parameters of the general model \eqref{eq: inf main} from discrete observations, although computational efficiency remains a challenge.

For completeness, we also mention the work in \cite{Saussereau14, Comte19}, where the problem is framed as a non-parametric estimation problem with the goal of estimating $f(x)$ for arbitrary $x$ from discrete observations. The authors propose a Nadaraya--Watson estimator and prove its consistency.

\subsubsection{Towards general models with multiplicative noise}

Statistical inference for the general model \eqref{eq: inf main} with multiplicative noise, that is, with non-constant $\sigma_\theta$, is still in its infancy. The first paper to address multiplicative noise is \cite{Chronopoulou13}. The authors propose using a pseudo-log-likelihood function
\[
\ell_n(\theta) = \sum_{i=1}^n \log p_\theta(t_i,X_{t_i}),
\]
where $p_\theta$ denotes the marginal distribution of $X_{t_i}$ under parameter $\theta$. Assuming that $H>\frac{1}{2}$, they derive an equivalent expression for the score function $\nabla_\theta \ell_n(\theta)$ that can be computed using Malliavin derivatives and Skorohod integrals, with further approximation by Monte Carlo methods and Euler-type discretisations. Although they do not obtain theoretical consistency results, numerical experiments illustrate the potential of the method.

The authors of \cite{Chronopoulou13} acknowledge a major obstacle to proving consistency: the lack of general results on the existence and uniqueness of an invariant distribution for \eqref{eq: inf main} in the multiplicative case. This issue also arises when attempting to extend the methodologies of \cite{Panloup20, Haress24}. Moreover, the methods in \cite{Neuenkirch14, Panloup20, Haress24, Chronopoulou13} rely on Euler discretisation, whose strong convergence typically fails in the multiplicative-noise setting when $H<\frac{1}{2}$.

Below, we present a different approach to constructing an approximate likelihood for the observations \eqref{eq: observations}. As with many methods at an early stage of development, we focus first on the one-dimensional fractional Ornstein--Uhlenbeck model \eqref{eq: fOU}. However, the same approach can be extended to construct the likelihood for solutions of the general model \eqref{eq: inf main} \cite{MorrishPapavasiliou26}. By formulating the problem first as an inverse problem, one can circumvent both the difficulties associated with ergodic limits for \eqref{eq: inf main} and the challenges of controlling the discretisation error.
\subsection{Likelihood construction and the inverse problem}

Suppose that we are given discrete observations \eqref{eq: observations} of a trajectory of the solution to \eqref{eq: fOU}. Rather than attempting to construct the likelihood $L(\theta,\sigma \mid x_{{\mathcal D}_{\delta,T}})$ directly, we instead ask the following question: given \eqref{eq: observations} and conditioned on $(\theta,\sigma)$, can we identify a trajectory of $B^H=\{B_t^H:t\in[0,T]\}$, denoted by
\[
B^H\bigl(\omega_{\theta,\sigma}(x_{{\mathcal D}_{\delta,T}})\bigr),
\]
that reproduces the observed data when used as the driving signal in \eqref{eq: fOU}? If so, then the observations can be expressed as a function of $B^H(\omega_{\theta,\sigma})$, and one may exploit the fact that the distribution of $B^H$ is known, thereby enabling the construction of the likelihood associated with \eqref{eq: observations}.

\subsubsection{The inverse problem}

Finding such a trajectory $B_t^H(\omega_{\theta,\sigma})$ requires solving a type of inverse problem, as studied in \cite{Morrish26}. The construction begins by considering finite-dimensional approximations of the driving path in \eqref{eq: inf main} that converge in $p$-variation. In the case of fractional Brownian motion $B^H$, it is well known that, for $H>\frac{1}{4}$, piecewise linear interpolations on nested dyadic partitions converge in $p$-variation \cite{Lyons02}, thereby defining the $p$-rough path lift that drives \eqref{eq: inf main}. A natural starting point is therefore the family of continuous piecewise linear paths on the partition ${\mathcal D}_\delta$, denoted by $B^{H,\delta}(\mathbf{c}^\delta)$, for $\mathbf{c}^\delta=(c_1^\delta,\dots,c_N^\delta)$ and $N\delta=T$. That is, for $t\in[(k-1)\delta,k\delta]$ and $k=1,\dots,N$,
\begin{equation}
\label{eq: pwl fBM}
B^{H,\delta}(\mathbf{c}^\delta)_t
=
B^{H,\delta}(\mathbf{c}^\delta)_{(k-1)\delta}
+
c_k^\delta\bigl(t-(k-1)\delta\bigr).
\end{equation}

In this context, the result of \cite{Lyons02} may be interpreted as follows: given a trajectory $B^H(\omega)$ with $H>\frac{1}{4}$ and a sequence of nested dyadic partitions ${\mathcal D}_{\delta_n,T}$, where $\delta_0>0$ and $\delta_n=2^{-n}\delta_0$, there exists a choice of parameters
\[
\mathbf{c}^{\delta_n}_B(\omega)
=
\bigl(c^{\delta_n}_{B,1}(\omega),\dots,c^{\delta_n}_{B,N_n}(\omega)\bigr)
\]
given by
\begin{equation}
\label{eq: pwl calibration}
c^{\delta_n}_{B,k}(\omega)
=
\frac{B^H(\omega)_{k\delta_n}-B^H(\omega)_{(k-1)\delta_n}}{\delta_n},
\qquad
k=1,\dots,N_n=\left\lfloor \frac{T}{\delta_n}\right\rfloor,
\end{equation}
such that the corresponding sequence $B^{H,\delta_n}(\mathbf{c}^{\delta_n}_B(\omega))$ converges to $B^H(\omega)$ in the $p$-rough path topology, with probability $1$.

The novelty here is that the approximation has been broken down into a two-step process: first, one fixes a finite-dimensional parametric family of paths $B^{H,\delta}(\mathbf{c}^\delta)$ (the approximate model choice step), and then one calibrates the parameter $\mathbf{c}^{\delta_n}$ to the trajectory $B^H(\omega)$ using \eqref{eq: pwl calibration} (the calibration step). The second step may then be interpreted independently as the construction of a calibration function
\[
\boldsymbol{\phi}_\delta:\Omega\to \mathbb{R}^N,
\]
such that $\mathbf{c}^\delta=\boldsymbol{\phi}_\delta(\omega)$ once the approximate model choice has been made. Here $\boldsymbol{\phi}_\delta(\omega)$ is adapted to the information available, with that information being $B^H(\omega)$ observed on the partition ${\mathcal D}_{\delta_n,T}$, denoted by $B^H(\omega)_{{\mathcal D}_{\delta_n,T}}$ in the case of the calibration given by \eqref{eq: pwl calibration}.

The advantage of the two-step approach is that it allows the calibration to be extended to other types of information while keeping the same approximate model choice. In particular, in order to construct an approximate solution to the inverse problem, one may ask the following question: suppose that one observes
\[
x^\delta_{{\mathcal D}_{\delta,T}}
=
\{X^\delta(\omega)_{t_i}: t_i\in{\mathcal D}_{\delta,T}\},
\]
where $X^\delta(\omega)$ is the solution to
\begin{equation}
\label{eq: approx fOU}
\mathrm{d}X^\delta(\omega)_t
=
-\theta X^\delta(\omega)_t\,\mathrm{d}t
+
\sigma\,\mathrm{d}B^{H,\delta}(\mathbf{c}^\delta(\omega))_t.
\end{equation}
For a fixed $\omega$, can one find $B^{H,\delta}(\mathbf{c}^\delta(\omega))$ such that the solution to \eqref{eq: approx fOU} matches the observations $x^\delta_{{\mathcal D}_{\delta,T}}$? This leads to the construction of a different calibration function for $\mathbf{c}^\delta$, mapping $X^\delta(\omega)_{{\mathcal D}_{\delta,T}}$ to $\mathbb{R}^N$.

To construct this calibration function, note first that since $B^{H,\delta}(\mathbf{c}^\delta)$ is assumed to be piecewise linear, \eqref{eq: approx fOU} becomes
\[
\mathrm{d}X_t^\delta
=
-\theta X_t^\delta\,\mathrm{d}t
+
\sigma c_k^\delta\,\mathrm{d}t
=
\bigl(-\theta X_t^\delta+\sigma c_k^\delta\bigr)\,\mathrm{d}t,
\]
for $t\in[(k-1)\delta,k\delta]$, with initial condition
\[
X^\delta_{(k-1)\delta}=x^\delta_{(k-1)\delta}
\]
and terminal condition
\[
X^\delta_{k\delta}=x^\delta_{k\delta},
\]
as determined by the observations. This allows one to solve for $c_k^\delta$. Indeed, if $F(t;X_0,c)$ is the solution to the generic ODE
\[
\mathrm{d}Z_t = \bigl(-\theta Z_t+\sigma c\bigr)\,\mathrm{d}t,
\]
with initial condition $Z_0=X_0$, then $c_k^\delta$ is the solution, with respect to $c$, of the equation
\[
x^\delta_{k\delta}=F\bigl(\delta;x^\delta_{(k-1)\delta},c\bigr).
\]

Up to this point, the specific assumption that the underlying model for $X$ is the fOU process \eqref{eq: fOU} has not been used. Although this is not necessary in general, in the present case it yields a convenient explicit form for $F(t;X_0,c)$, namely
\[
F(t;X_0,c)
=
e^{-\theta t}X_0
+
\bigl(1-e^{-\theta t}\bigr)\frac{\sigma c}{\theta}.
\]
Moreover, the equation
\[
x^\delta_{k\delta}=F\bigl(\delta;x^\delta_{(k-1)\delta},c\bigr)
\]
has an explicit solution, which leads to the calibration function
\begin{align}
\label{eq: inv calibration}
\widetilde{c}_k^\delta(\omega)
=
\widetilde{\phi}_\delta(x^\delta_{(k-1)\delta},x^\delta_{k\delta})
&=
\frac{\bigl(x^\delta_{k\delta}-x^\delta_{(k-1)\delta}e^{-\theta\delta}\bigr)\theta}{\sigma(1-e^{-\theta\delta})}
\nonumber\\
&=
\frac{\bigl(X^\delta(\omega)_{k\delta}-X^\delta(\omega)_{(k-1)\delta}e^{-\theta\delta}\bigr)\theta}{\sigma(1-e^{-\theta\delta})}.
\end{align}

For discrete observations on the partition ${\mathcal D}_\delta$, $B^H$ has therefore been approximated by a piecewise linear path on ${\mathcal D}_\delta$ (not necessarily a piecewise linear interpolation), which is necessary in order to restrict the number of degrees of freedom to the dimension of the information space. This approximation assumption is equivalent to the Euler discretisation used by other methods, but it allows for better control of the error because the discretisation is now used to approximate $B^H$ itself rather than the solution $X$. Note that while the calibration \eqref{eq: inv calibration} is exact if the data come from observing $X^\delta$ in \eqref{eq: approx fOU}, the observations of interest come from \eqref{eq: fOU}. The corresponding calibration constant then becomes
\begin{equation}
\label{eq: inv calibration approx}
c^\delta_{x,k}(\omega)
=
\widetilde{\phi}_\delta(x_{(k-1)\delta},x_{k\delta})
=
\frac{\bigl(X(\omega)_{k\delta}-X(\omega)_{(k-1)\delta}e^{-\theta\delta}\bigr)\theta}{\sigma(1-e^{-\theta\delta})}.
\end{equation}

The approximation is meaningful if one can show that $B^{H,\delta_n}(\mathbf{c}^{\delta_n}_x(\omega))$ converges almost surely in the $p$-rough path topology to $B^H(\omega)$, where $x=X(\omega)$ is the solution to \eqref{eq: fOU} driven by $B^H(\omega)$. In \cite{Morrish26}, this is shown for the general model under appropriate assumptions. Below, a straightforward proof is given for the fOU case.

\begin{theorem}
\label{thm: inv continuity}
Suppose that $X(\omega):[0,T]\to\mathbb{R}$ is the solution to \eqref{eq: fOU} driven by $B^H(\omega)$, with $H>\frac{1}{4}$. Let $B^{H,\delta_n}(\mathbf{c}^{\delta_n}_x(\omega))$ be the piecewise linear path on ${\mathcal D}_{\delta_n,T}$ with $\delta_n=2^{-n}\delta_0$ for some $\delta_0>0$, given by \eqref{eq: pwl fBM}, with gradients $\mathbf{c}^{\delta_n}_x(\omega)$ given by \eqref{eq: inv calibration approx}. Then, as $n\to\infty$,
\[
d_{p\text{-}\mathrm{var}}\bigl({\mathbb B}^H(\omega),{\mathbb B}^{H,\delta_n}(\mathbf{c}^{\delta_n}_x(\omega))\bigr)\to 0
\]
almost surely, where $d_{p\text{-}\mathrm{var}}$ denotes the $p$-rough path metric, ${\mathbb B}^H$ is the $p$-rough path lift defined as the limit of piecewise linear interpolations of $B^H(\omega)$ and ${\mathbb B}^{H,\delta_n}$ is the $p$-rough path lift of the bounded variation path ${B}^{H,\delta_n}$.
\end{theorem}

\begin{proof}
First, since the piecewise linear interpolations on nested dyadic intervals $B^{H,\delta_n}(\mathbf{c}^{\delta_n}_B(\omega))$ converge almost surely in the $p$-rough path topology for $H>\frac{1}{4}$, it is sufficient to show that
\[
d_{p\text{-}\mathrm{var}}\bigl({\mathbb B}^{H,\delta_n}(\mathbf{c}^{\delta_n}_B(\omega)),{\mathbb B}^{H,\delta_n}(\mathbf{c}^{\delta_n}_x(\omega))\bigr)\to 0
\]
as $n\to\infty$ \cite{Lyons02}.

Since $B^H$ is one-dimensional, to prove that two sequences of piecewise linear paths on the same partitions converge in the $p$-rough path topology, it is sufficient to show that, almost surely,
\begin{equation}
\label{eq: p-var bound}
\|B^{H,\delta_n}(\mathbf{c}^{\delta_n}_B(\omega))\|_{p\text{-}\mathrm{var}}<+\infty
\end{equation}
uniformly in $n$, and that the difference converges to zero in $p$-variation, that is,
\begin{equation}
\label{eq: p-var difference}
\|B^{H,\delta_n}(\mathbf{c}^{\delta_n}_B(\omega))-B^{H,\delta_n}(\mathbf{c}^{\delta_n}_x(\omega))\|_{p\text{-}\mathrm{var}}\to 0
\end{equation}
as $n\to\infty$. Note that convergence in $p$-variation under the norm $\|\cdot\|_{p\text{-}\mathrm{var}}$ is different from convergence in the $p$-rough path topology under the metric $d_{p\text{-}\mathrm{var}}$, which also includes convergence of higher iterated integrals.

It is already known that
\[
\|B^{H,\delta_n}(\mathbf{c}^{\delta_n}_B(\omega))\|_{p\text{-}\mathrm{var}}<+\infty
\]
almost surely for any $p>\frac{1}{H}$, since $B^{H,\delta_n}(\mathbf{c}^{\delta_n}_B(\omega))$ converges to $B^H$. The standard construction of the solution to \eqref{eq: fOU} gives
\begin{equation}
\label{eq: fOU solution}
X(\omega)_{k\delta}
=
X(\omega)_{(k-1)\delta}e^{-\theta\delta}
+
\sigma\int_{(k-1)\delta}^{k\delta} e^{-\theta(k\delta-s)}\,\mathrm{d}B^H(\omega)_s.
\end{equation}
Substituting this into \eqref{eq: inv calibration approx} yields
\begin{align*}
c^\delta_{x,k}(\omega)
&=
\frac{\theta}{1-e^{-\theta\delta}}
\int_{(k-1)\delta}^{k\delta} e^{-\theta(k\delta-s)}\,\mathrm{d}B^H(\omega)_s
\\
&=
\frac{\theta}{1-e^{-\theta\delta}}
\int_{(k-1)\delta}^{k\delta} \mathrm{d}B^H(\omega)_s
-
\frac{\theta}{1-e^{-\theta\delta}}
\int_{(k-1)\delta}^{k\delta}
\bigl(1-e^{-\theta(k\delta-s)}\bigr)\,\mathrm{d}B^H(\omega)_s.
\end{align*}
Hence the normalised difference between the gradients is
\begin{align*}
\delta\bigl(c^\delta_{B,k}(\omega)-c^\delta_{x,k}(\omega)\bigr)
&=
\left(1-\frac{\theta\delta}{1-e^{-\theta\delta}}\right)\Delta_\delta B^H(\omega)_k
\\
&\quad
-
\frac{\theta\delta}{1-e^{-\theta\delta}}
\int_{(k-1)\delta}^{k\delta}
\bigl(1-e^{-\theta(k\delta-s)}\bigr)\,\mathrm{d}B^H(\omega)_s,
\end{align*}
where
\[
\Delta_\delta B^H(\omega)_k
=
B^H(\omega)_{k\delta}-B^H(\omega)_{(k-1)\delta}.
\]
It follows that
\[
\delta\bigl(c^\delta_{B,k}(\omega)-c^\delta_{x,k}(\omega)\bigr)
=
\mathcal{O}\bigl(\delta^{1+\frac{1}{p}}\bigr).
\]
Using \eqref{eq: p-var bound}, this yields \eqref{eq: p-var difference}.
\end{proof}

\subsubsection{Likelihood construction}

Given observations $x_{{\mathcal D}_{\delta,T}}$ as in \eqref{eq: observations}, let $B^{H,\delta}(\mathbf{c}^\delta_x(\omega))$ be the corresponding approximate solution to the inverse problem constructed above. There are two ways to interpret this object: (i) as an approximation to the piecewise linear interpolation on ${\mathcal D}_{\delta,T}$ of the exact fractional Brownian motion path $B^H(\omega)$ corresponding to the observations $x_{{\mathcal D}_{\delta,T}}$, in the sense of Theorem \ref{thm: inv continuity}; or (ii) as the exact solution to the inverse problem corresponding to the approximate model \eqref{eq: approx fOU}, where the observations $x_{{\mathcal D}_{\delta,T}}$ are now viewed as approximations to the corresponding exact observations $x^\delta_{{\mathcal D}_{\delta,T}}$. In other words, $B^{H,\delta}(\mathbf{c}^\delta_x(\omega))$ may be regarded either as the approximate solution to the inverse problem for the exact model or as the exact solution to the inverse problem for an approximate model.

These two interpretations lead to two corresponding approaches to constructing an approximation to the likelihood of the exact discrete observations $x_{{\mathcal D}_{\delta,T}}$ of the full model \eqref{eq: fOU}: either one builds an approximate likelihood for the full model and plugs in data generated by that model, or one builds an exact likelihood for the approximate model \eqref{eq: approx fOU} and plugs in observations that are now approximate from the point of view of that model. Under appropriate assumptions, both approaches can yield approximate likelihoods that converge, in some sense, to the exact likelihood as $\delta\to 0$.

Here the latter approach is adopted. Suppose that one is given discrete observations
\[
x^\delta_{{\mathcal D}_{\delta,T}}
=
X^\delta(\omega)_{{\mathcal D}_{\delta,T}}
\]
on the partition ${\mathcal D}_{\delta,T}$, where $X^\delta(\omega)$ solves \eqref{eq: approx fOU} for a piecewise linear path $B^{H,\delta}$ on the same partition, with increments distributed as the increments of fractional Brownian motion on ${\mathcal D}_{\delta,T}$. Generalising \eqref{eq: fOU solution}, the observations $X^\delta(\omega)_{{\mathcal D}_{\delta,T}}$ can be written as a function of the vector of increments:
\[
X^\delta(\omega)_{{\mathcal D}_{\delta,T}}
=
I^\delta_{\theta,\sigma}\bigl(\Delta_\delta B^H(\omega)\bigr)
=
I^\delta_{\theta,\sigma}\bigl(\delta\mathbf{c}^\delta_B(\omega)\bigr),
\]
where $I^\delta_{\theta,\sigma}$ denotes the It\^o map. As seen above, this map can be inverted to give
\[
\Delta_\delta B^H(\omega)
=
\delta\mathbf{c}^\delta_B(\omega)
=
I_{\theta,\sigma}^{-1,\delta}\bigl(X^\delta(\omega)_{{\mathcal D}_{\delta,T}}\bigr).
\]
Thus, $X^\delta(\omega)_{{\mathcal D}_{\delta,T}}$ may be viewed as a transformation of the finite-dimensional vector $\mathbf{c}^\delta_B(\omega)$. Since the distribution of $\delta\mathbf{c}^\delta_B(\omega)$ is assumed to coincide with that of the increments of fractional Brownian motion, the distribution of $X^\delta(\omega)_{{\mathcal D}_{\delta,T}}$ conditioned on $(\theta,\sigma)$, that is, the likelihood, may be expressed in terms of the distribution of the increments of fractional Brownian motion. For fixed $T>0$, one obtains
\begin{eqnarray*}
L_\delta(x^\delta_{{\mathcal D}_{\delta,T}}\mid \theta,\sigma)
&=&
|2\pi\Sigma_\delta^H|^{-1/2}
\exp\left(
-\frac{1}{2}
{I_{\theta,\sigma}^{-1,\delta}(x^\delta_{{\mathcal D}_{\delta,T}})}^*
(\Sigma_\delta^H)^{-1}
I_{\theta,\sigma}^{-1,\delta}(x^\delta_{{\mathcal D}_{\delta,T}})
\right)\\
&& \cdot
\left|J\bigl(I_{\theta,\sigma}^{-1,\delta}(x^\delta_{{\mathcal D}_{\delta,T}})\bigr)\right|,
\end{eqnarray*}
where $*$ denotes matrix transpose, $\Sigma_\delta^H$ is the covariance matrix of the increments of $B^H(\omega)$ on ${\mathcal D}_{\delta,T}$, and $J$ is the Jacobian correction. Note that $I_{\theta,\sigma}^{-1,\delta}(x^\delta_{{\mathcal D}_{\delta,T}})$ is precisely the solution to the inverse problem, which in the case of \eqref{eq: approx fOU} is given by \eqref{eq: inv calibration}. The parameter-dependent part of the log-likelihood is therefore
\begin{align}
\label{eq: approx log-like}
\ell_\delta(x^\delta_{{\mathcal D}_{\delta,T}}\mid \theta,\sigma)
&\propto
-\frac{1}{2}
\left(
{I_{\theta,\sigma}^{-1,\delta}(x^\delta_{{\mathcal D}_{\delta,T}})}^*
(\Sigma_\delta^H)^{-1}
I_{\theta,\sigma}^{-1,\delta}(x^\delta_{{\mathcal D}_{\delta,T}})
\right)
+
\log\left|
J\bigl(I_{\theta,\sigma}^{-1,\delta}(x^\delta_{{\mathcal D}_{\delta,T}})\bigr)
\right|
\nonumber\\
&=
-\frac{1}{2\sigma^2\phi_\delta(\theta)^2}
(\Delta_{\delta,\theta}x^\delta)^*
(\Sigma_\delta^H)^{-1}
(\Delta_{\delta,\theta}x^\delta)
-
N\log\bigl(\sigma\phi_\delta(\theta)\bigr),
\end{align}
where
\[
\phi_\delta(\theta)=\frac{1-e^{-\theta\delta}}{\theta\delta},
\qquad
(\Delta_{\delta,\theta}x^\delta)_k
=
x^\delta_{k\delta}-x^\delta_{(k-1)\delta}e^{-\theta\delta}.
\]

This is the exact likelihood for the approximate model \eqref{eq: approx fOU}, but an approximate likelihood for observations $x_{{\mathcal D}_{\delta,T}}$ from the exact model. Below, we discuss the asymptotic behaviour and, in particular, the consistency of the corresponding estimators.

\subsubsection{Asymptotic behaviour of approximate likelihood.}
First, a closer inspection of \eqref{eq: approx log-like} suggests that an appropriate scaling is needed before studying the limit $\delta\to 0$. Indeed, the following result holds.

\begin{proposition}[{\cite{MorrishPapavasiliou26}}]
\label{prop: log-like expansion}
Suppose that $x^\delta=X^\delta(\omega)$ is a realisation of a trajectory of a solution to \eqref{eq: approx fOU}. Then the log-likelihood $\ell_\delta(x^\delta_{{\mathcal D}_{\delta,T}}\mid \theta,\sigma)$ given in \eqref{eq: approx log-like} can be written as
\[
\ell_\delta(x^\delta_{{\mathcal D}_{\delta,T}}\mid \theta,\sigma)
=
\frac{1}{\delta}\ell_{0,\delta}(x^\delta_{{\mathcal D}_{\delta,T}}\mid \sigma)
+
\ell_{1,\delta}(x^\delta_{{\mathcal D}_{\delta,T}}\mid \theta,\sigma)
+
\mathcal{O}(\delta),
\]
where
\[
\ell_{0,\delta}(x^\delta_{{\mathcal D}_{\delta,T}}\mid \sigma)
:=
-\frac{T}{2\sigma^2N}
(\Delta x^\delta_{{\mathcal D}_{\delta,T}})^*
(\Sigma_\delta^H)^{-1}
(\Delta x^\delta_{{\mathcal D}_{\delta,T}})
-
T\log(\sigma),
\]
and
\begin{align*}
\ell_{1,\delta}(x^\delta_{{\mathcal D}_{\delta,T}}\mid \theta,\sigma)
:=
-\frac{\theta T}{2\sigma^2}
\Biggl(
&
\frac{(\Delta x^\delta_{{\mathcal D}_{\delta,T}})^*(\Sigma_\delta^H)^{-1}(\Delta x^\delta_{{\mathcal D}_{\delta,T}})}{N}
\\
&\quad
+
\frac{2(\Delta x^\delta_{{\mathcal D}_{\delta,T}})^*(\Sigma_\delta^H)^{-1}(S_{-1}x^\delta_{{\mathcal D}_{\delta,T}})}{N}
\\
&\quad
+
\frac{\theta T (S_{-1}x^\delta_{{\mathcal D}_{\delta,T}})^*(\Sigma_\delta^H)^{-1}(S_{-1}x^\delta_{{\mathcal D}_{\delta,T}})}{N^2}
-
\sigma^2
\Biggr),
\end{align*}
with $S_{-1}$ denoting the backward shift operator, that is,
\[
(S_{-1}x^\delta_{{\mathcal D}_{\delta,T}})_k=x^\delta_{(k-1)\delta}.
\]
\end{proposition}

It follows from Proposition \ref{prop: log-like expansion} that the $\mathcal{O}(\delta^{-1})$ term depends only on the diffusion parameter $\sigma$. This is not surprising, since processes with different diffusion coefficients are not absolutely continuous with respect to one another.

Suppose now that $(\theta_0,\sigma_0)$ are the true parameter values. Then the solution to the inverse problem
\[
I_{\theta,\sigma}^{-1,\delta}(x^\delta_{{\mathcal D}_{\delta,T}})
\]
has the distribution of the increments of fractional Brownian motion. Using this fact, one may prove consistency in a weak sense, namely that the derivatives of the appropriately scaled log-likelihood vanish at the true parameter values. The proof requires careful control of the behaviour of the covariance matrix of fractional Gaussian noise and is based on the following.

\begin{conjecture}
    \label{conjecture}
Let $\Sigma^H_{1,N}$ be the covariance matrix of $N$ unit increments of fractional Brownian motion and $P_k$ be the $N\times 2N$ matrix given by
\begin{equation}
\label{eq: projection matrix P}
   P_k = \left( {\bf 0}_{N\times k}, I_{N}, {\bf 0}_{N\times (N-k)}\right),  
\end{equation}
which projects a $2N\times 1$ vector to the sub-vector with coordinates $(k+1,N+k)$. Set 
\[
A_k = (\Sigma^{H}_{1,N})^{-1}P_k \Sigma^H_{1,2N}P_0^*.
\]
Then, trace $Tr(A_k)$ and $Tr(A_k A_l)$ are uniformly bounded with respect to $N$ and $k,l=1,\dots,N$.
\end{conjecture}
\begin{remark}
    We believe that the conjecture is true, based on numerical evidence and also heuristic computations, using Toeplitz matrix properties. If $T_N(f_H)$ is the Toeplitz matrix generated by the spectral density of fractional Gaussian noise $f_H$, then $A_k$ can be written as
    \[
    A_k = T_N(f_H)^{-1} T_N(e^{-ik\lambda}f_H) \approx T_N(f_H)^{-1} T_N(f_H) T_N(e^{-ik\lambda}) = T_N(e^{-ik\lambda}),
    \]
    whose trace is trivially $0$. However, the standard assumptions of the Toeplitz product limits to hold do not apply here.
\end{remark}

This leads to the following.
\begin{lemma}
\label{lemma: trace bound}
Let $H\in(0,1)$ and $\{ B^H(\omega)_t | -\infty <t\leq T\}$ a fractional Brownian motion with infinite time horizon. Let $\Delta_\delta B^H(\omega)$ be the increments of fractional Brownian motion on $\{0,\dots,(N-1)\delta,N\delta\}$ with covariance function $\Sigma^H_{\delta,N}$, for $\delta>0$. Set 
\[
Q^{k,0}_N = \bigl(S_{-k}(\Delta_\delta B^H(\omega))\bigr)^*
(\Sigma_{\delta,N}^H)^{-1}(\Delta_\delta B^H(\omega)),
\]
where $S_{-k}$ is the shift operator, i.e. for vector ${z_i}_{i=-N+1}^N$, $(S_k z)_i = z_{i-k}$ defined for $i=1,\dots,N$. Then, under conjecture \ref{conjecture}, there exists a $d_H\in \R$ depending only on $H$ such that
\[
|{\mathbb E}\left( Q^{k,0}_N Q^{l,0}_N\right)| < d_H + N\delta_{kl}, \forall N>0\ {\rm and}\ k,l=1,\dots,N
\]
where $\delta_{kl}$ is Kr\"onecker's delta.
\end{lemma}

\begin{proof}
Let $H\in(0,1)$ and $\delta>0$. For each $N\geq 1$, let $\Sigma^H_{\delta,2N}$ be the covariance matrix of the $2N$ consecutive increments of fBM and let $\Sigma^H_{\delta,N}$ be the covariance matrix of any block of $N$ consecutive increments. For $0\leq k\leq N$, define
\[
\zeta^{(k)}
:=
(\Sigma^H_{\delta,N})^{-1/2}\,
P_k\,
(\Sigma^H_{\delta,2N})^{1/2}\,\xi,
\]
where $\xi\sim \mathcal N(0,I_{2N})$. Set
\[
Q_{N}^{k,l} = \sum_{i=1}^N \zeta_i^{(k)}\zeta_i^{(l)}
\]
and $G_{N}^{k,l}=\operatorname{Cov}(\zeta^{(k)},\zeta^{(l)})$. Then
\begin{align*}
  G_{N}^{k,l} &= {\mathbb E}\left( \zeta^{(k)}(\zeta^{(l)})^*\right) 
  \\
  \qquad
  &= (\Sigma^H_{\delta,N})^{-1/2}\,
P_k\,
(\Sigma^H_{\delta,2N})^{1/2}\,\mathbb{E}\left( \xi \xi^* \right)(\Sigma^H_{\delta,2N})^{1/2,*} P_l^* (\Sigma^H_{\delta,N})^{-1/2,*} 
\\
\qquad
&= (\Sigma^H_{\delta,N})^{-1/2}\,
P_k\,
(\Sigma^H_{\delta,2N}) P_l^* (\Sigma^H_{\delta,N})^{-1/2,*} = (\Sigma^H_{\delta,N})^{-1/2}C_{N}^{k,l}
 (\Sigma^H_{\delta,N})^{-1/2,*},
\end{align*}
where $C_{N}^{k,l}:=P_{k}\,\Sigma^H_{\delta,2N}\,P_{l}^*$ and $P_k$ defined in \eqref{eq: projection matrix P}. To compute ${\mathbb E}\left( Q^{k,0}_N Q^{l,0}_N\right)$, we apply Wick's formula, which yields:
\begin{align*}
{\mathbb E}\left( Q^{k,0}_N Q^{l,0}_N \right) &= {\mathbb E}\left( \sum_{i,j=1}^N \zeta^{(k)}_i \zeta^{(0)}_i \zeta^{(l)}_j \zeta^{(0)}_j \right) = \sum_{i,j=1}^N {\mathbb E}\left(  \zeta^{(k)}_i \zeta^{(0)}_i \zeta^{(l)}_j \zeta^{(0)}_j \right)
    \\
    \qquad
    &= 
    \sum_{i,j=1}^N{\mathbb E}\left( \zeta^{(k)}_i \zeta^{(0)}_i \right){\mathbb E}\left( \zeta^{(l)}_j \zeta^{(0)}_j\right)
    \\
    \qquad
    &+\sum_{i,j=1}^N{\mathbb E}\left( \zeta^{(k)}_i \zeta^{(l)}_j \right){\mathbb E}\left( \zeta^{(0)}_i \zeta^{(0)}_j\right)
    +\sum_{i,j=1}^N{\mathbb E}\left( \zeta^{(k)}_i \zeta^{(0)}_j\right){\mathbb E}\left( \zeta^{(0)}_i \zeta^{(l)}_j \right)
    \\
    \qquad
    &= 
    \sum_{i,j=1}^N (G^{k,0}_N)_{ii}(G^{l,0}_N)_{jj} +\sum_{i,j=1}^N (G^{k,l}_N)_{ij}\delta_{ij}
    +\sum_{i,j=1}^N (G^{k,0}_N)_{ij}(G^{l,0}_N)_{ji}.
\end{align*}
It follows that
\[
{\mathbb E}\left( Q^{k,0}_N Q^{l,0}_N\right) = Tr\left(G^{k,0}_N\right)Tr\left(G^{l,0}_N\right) + Tr\left(G^{|k-l|,0}_N\right) + Tr\left(G^{k,0}_N G^{l,0}_N\right),
\]
where we used the property $G^{k,l}_N = G^{k-l,0}_N$ for $k\geq l$ or $G^{k,l}_N = (G^{l-k,0}_N)^* = G^{l-k,0}_N$ for $k<l$, with $Tr(\cdot)$ denoting the trace. 
Clearly, for $k=0$, $C^{0,0}_N = \Sigma^H_{\delta,N}$, so $Tr\left((\Sigma^H_{\delta,N})^{-1}C_{N}^{0,0}\right) = N$.The result follows from conjecture \ref{conjecture}, self-similarity and the cyclic invariance of the trace. 
\end{proof}

We can now prove asymptotic consistency, in the following sense.
\begin{theorem}
\label{thm: consistency of approx model}
Suppose that $x^\delta=X^\delta(\omega)$ is a realisation of a trajectory of a solution to \eqref{eq: approx fOU} with parameter values $(\theta_0,\sigma_0)$ and $x_0 = \sigma_0\int_{-\infty}^0 e^{\theta_0 s}dB^{H,\delta}(\omega)_s$, i.e. it comes from the invariant distribution. Then, for $\ell_\delta(x^\delta_{{\mathcal D}_{\delta,T}}\mid \theta,\sigma)$ given by \eqref{eq: approx log-like}, the following results hold almost surely, as $\delta \to 0$, for $T>0$ fixed:
\[
\partial_\sigma\bigl(\delta\ell_\delta(x^\delta_{{\mathcal D}_{\delta,T}}\mid \theta,\sigma)\bigr)\big|_{\theta=\theta_0,\sigma=\sigma_0}\to 0,
\]
and, assuming conjecture \ref{conjecture} holds, 
\[
\partial_\theta(\delta\ell_\delta(x^\delta_{{\mathcal D}_{\delta,T}}\mid \theta,\sigma))\big|_{\theta=\theta_0,\sigma=\sigma_0}\to 0.
\]
\end{theorem}

\begin{proof}
First,
\begin{align*}
\partial_\sigma\bigl(\delta\ell_\delta(x^\delta_{{\mathcal D}_{\delta,T}}\mid \theta,\sigma)\bigr)
&=
\frac{\delta}{\sigma^3\phi_\delta(\theta)^2}
(\Delta_{\delta,\theta}x^\delta)^*
(\Sigma_\delta^H)^{-1}
(\Delta_{\delta,\theta}x^\delta)
-
\delta\frac{N}{\sigma}
\\
&=
\frac{T}{\sigma}
\left(
\frac{1}{N}
\bigl(I_{\theta,\sigma}^{-1,\delta}(x^\delta_{{\mathcal D}_{\delta,T}})\bigr)^*
(\Sigma_\delta^H)^{-1}
\bigl(I_{\theta,\sigma}^{-1,\delta}(x^\delta_{{\mathcal D}_{\delta,T}})\bigr)
-
1
\right).
\end{align*}
As discussed above, for the true parameter values $(\theta_0,\sigma_0)$ one has
\[
I_{\theta_0,\sigma_0}^{-1,\delta}(x^\delta_{{\mathcal D}_{\delta,T}})=\Delta B^H(\omega),
\]
and therefore
\[
\partial_\sigma\bigl(\delta\ell_\delta(x^\delta_{{\mathcal D}_{\delta,T}}\mid \theta_0,\sigma_0)\bigr)
=
\frac{T}{\sigma_0}
\left(
\frac{1}{N}\sum_{i=1}^N \xi_i^2 - 1
\right),
\]
where $\{\xi_i\}_{i=1}^N$ are i.i.d.\ $\mathcal N(0,1)$ random variables. The strong law of large numbers therefore implies that
\[
\partial_\sigma\bigl(\delta\ell_\delta(x^\delta_{{\mathcal D}_{\delta,T}}\mid \theta_0,\sigma_0)\bigr)\to 0
\]
almost surely as $N\to\infty$ or, equivalently, $\delta\to 0$. 

Similarly,
\begin{align}
\label{eq: log like theta derivative}
\partial_\theta(\delta\ell_\delta(x^\delta_{{\mathcal D}_{\delta,T}}\mid \theta,\sigma))
&=
\frac{\delta\phi_\delta'(\theta)}{\sigma^2\phi_\delta(\theta)^3}
(\Delta_{\delta,\theta}x^\delta)^*
(\Sigma_\delta^H)^{-1}
(\Delta_{\delta,\theta}x^\delta)
-
\frac{\delta N\phi_\delta'(\theta)}{\phi_\delta(\theta)}
\nonumber\\
&\quad
-
\frac{\delta}{2\sigma^2\phi_\delta(\theta)^2}
\partial_\theta(\Delta_{\delta,\theta}x^\delta)^*
(\Sigma_\delta^H)^{-1}
(\Delta_{\delta,\theta}x^\delta)
\nonumber\\
&\quad
-
\frac{\delta}{2\sigma^2\phi_\delta(\theta)^2}
(\Delta_{\delta,\theta}x^\delta)^*
(\Sigma_\delta^H)^{-1}
\partial_\theta(\Delta_{\delta,\theta}x^\delta).
\end{align}
It is straightforward to verify that
\[
\frac{\phi_\delta'(\theta)}{\phi_\delta(\theta)}=-\frac{\theta}{2}+{\mathcal O}(\delta).
\]
Evaluating the first line on the right-hand side of \eqref{eq: log like theta derivative} at $(\theta_0,\sigma_0)$ gives
\begin{align*}
&-\frac{\theta T}{2N}\bigl(1+\mathcal{O}(\delta)\bigr)
\bigl(I_{\theta_0,\sigma_0}^{-1,\delta}(x^\delta_{{\mathcal D}_{\delta,T}})\bigr)^*
(\Sigma_\delta^H)^{-1}
I_{\theta_0,\sigma_0}^{-1,\delta}(x^\delta_{{\mathcal D}_{\delta,T}})
+
\frac{T\theta}{2}\bigl(1+\mathcal{O}(\delta)\bigr)
\\
&\qquad
=
-\frac{\theta T}{2}\bigl(1+\mathcal{O}(\delta)\bigr)
\left(
\frac{1}{N}\sum_{i=1}^N (\xi_i^2 - 1)
\right)
\to 0,
\end{align*}
almost surely, as $\delta\to 0$.

The last two terms are symmetric, so it is sufficient to show that one of them converges to zero. First, one can expand the derivative $\partial_\theta(\Delta_{\delta,\theta}x^\delta)$ as
\begin{eqnarray*}
\partial_\theta(\Delta_{\delta,\theta}x^\delta)
&=&
\partial_\theta\bigl(x^\delta_{{\mathcal D}_{\delta,T}}-e^{-\theta\delta}S_{-1}x^\delta_{{\mathcal D}_{\delta,T}}\bigr)
=
\delta e^{-\theta\delta}S_{-1}x^\delta_{{\mathcal D}_{\delta,T}}\\
&=&
\delta\sum_{k=1}^N e^{-k\theta\delta}S_{-k}(\Delta_{\delta,\theta}x^\delta),
\end{eqnarray*}
where $S_{-k}$ is the shift operator applied to the infinite vector of increments since $x_0$ comes from the invariant distribution -- hence $x^\delta_{k\delta} = e^{-\theta_0 \delta} x^\delta_{(k-1)\delta} + \sigma_0 \phi_\delta(\theta_0)\Delta_\delta B^{H}(\omega)_k$ also holds for $k\leq 0$. It follows that at $(\theta_0,\sigma_0)$, the third term in the right-hand-side of \eqref{eq: log like theta derivative} becomes
\begin{align}
\label{eq: like derivative term 2}
& -\frac{\delta^2}{2}\sum_{k=1}^N e^{-k\theta\delta}\frac{S_{-k}(\Delta_{\delta,\theta_0} x^\delta_{{\mathcal D}_{\delta,T}})^*}{\sigma_0\phi_\delta(\theta_0)}(\Sigma_\delta^H)^{-1}
\frac{(\Delta_{\delta,\theta_0} x^\delta_{{\mathcal D}_{\delta,T}})}{\sigma_0\phi_\delta(\theta_0)}
\nonumber
\\
&\qquad
= -\frac{\delta^2}{2}\sum_{k=1}^N e^{-k\theta\delta}\bigl(S_{-k}(I_{\theta_0,\sigma_0}^{-1,\delta}(x^\delta_{{\mathcal D}_{\delta,T}}))\bigr)^*
(\Sigma_\delta^H)^{-1}
I_{\theta_0,\sigma_0}^{-1,\delta}(x^\delta_{{\mathcal D}_{\delta,T}})
\nonumber
\\
&\qquad
= -\frac{T^2}{2N^2}\sum_{k=1}^N e^{-k\theta\delta}\bigl(S_{-k}(\Delta_\delta B^H(\omega))\bigr)^*
(\Sigma_\delta^H)^{-1}(\Delta_\delta B^H(\omega)).
\end{align}
Set $Q^{k,0}_{N} = \bigl(S_{-k}(\Delta_\delta B^H(\omega))\bigr)^*
(\Sigma_\delta^H)^{-1}(\Delta_\delta B^H(\omega))$. Then, using lemma \ref{lemma: trace bound}, the second moment of \eqref{eq: like derivative term 2} can be bounded by
\begin{align*}
    &\frac{T^4}{4N^4}\sum_{k,l=1}^N e^{-(k+l)\theta T/N} {\mathbb E}\left( Q^{k,0}_N Q^{l,0}_N\right) \leq 
    \\
    \qquad
    &\frac{T^4 d_H} {4N^4}\sum_{k\neq l, k,l=1}^N e^{-(k+l)\theta T/N} + \frac{T^4}{4N^4}\sum_{k=1}^N e^{-2k\theta T/N} N \sim{\mathcal O}\left(\frac{1}{N^2}\right).
\end{align*}
Almost sure convergence follows from the Borel-Cantelli lemma. 
\end{proof}

%% file: multiscale3.tex
\section{Fractional Diffusions as Limits of Multiscale Processes}
\label{sec:multiscale}

\subsection{Multiscale diffusions and the inference problem}

The previous section reviewed inference for fractional diffusion models treated as the primary objects of study. A complementary perspective arises when such models appear as effective limits of underlying multiscale systems, for which effective inference is central to their practical use. These models arise when several components of a system evolve on widely separated scales. In many cases, the components of interest evolve on the natural time scale of the problem, whereas the faster components represent environmental turbulence or perturbations. Rather than modelling these fast variables explicitly, it is often preferable to suppress them and derive an approximate model for the quantities of interest that retains only the non-negligible effects acting on the natural time scale.

The particular class of models considered here consists of systems of slow/fast SDEs with a time-scale separation between a slow component of interest \(X^{\varepsilon}\) and a fast variable \(Y^{\varepsilon}\), which is often either impossible or impractical to observe. A prototypical example of such a system is
\begin{align}
\label{eq: multiscale SDE system}
\begin{cases}
\mathrm{d}X^{\varepsilon}_t
=
\left(\frac{1}{\varepsilon}f_0(X^{\varepsilon}_t,Y^{\varepsilon}_t)+f_1(X^{\varepsilon}_t,Y^{\varepsilon}_t)\right)\mathrm{d}t
+
\frac{1}{\sqrt{\varepsilon}}g_1(X^{\varepsilon}_t,Y^{\varepsilon}_t)\mathrm{d}W^{1}_t,\\[0.3em]
\mathrm{d}Y^{\varepsilon}_t
=
\frac{1}{\varepsilon}f_2(X^{\varepsilon}_t,Y^{\varepsilon}_t)\mathrm{d}t
+
\frac{1}{\sqrt{\varepsilon}}g_2(X^{\varepsilon}_t,Y^{\varepsilon}_t)\mathrm{d}W^{2}_t,
\end{cases}
\end{align}
where \(W_t^1\) and \(W_t^2\) are two independent Brownian motions, and \(0<\varepsilon\ll 1\) is the time-scale separation parameter. Under suitable assumptions on the coefficients, it is possible to identify an approximate model \(\bar{X}\) such that \(X^{\varepsilon}\xrightarrow{\varepsilon\to 0}\bar{X}\) weakly, where \(\bar{X}\) solves an SDE of the form
\begin{equation}
\label{eq: effective SDE}
\mathrm{d}\bar{X}_t = \mu(\bar{X}_t)\mathrm{d}t + \sigma(\bar{X}_t)\mathrm{d}W_t.
\end{equation}
For details on the assumptions on the coefficients in \eqref{eq: multiscale SDE system} required for such a result, as well as on the dependence of the effective coefficients \(\mu\) and \(\sigma\) on \eqref{eq: multiscale SDE system}, the reader is referred to the monograph \cite{Pavliotis08}.

This yields a reduced model that appears to capture the relevant features of the quantity of interest. Moreover, the model is considerably more convenient from both statistical and computational perspectives, since it involves only an approximate version of the slow component and, in principle, no longer requires handling observations of the fast variable.

The next question concerns the usability of this model. In order to take advantage of the simpler effective dynamics, it is necessary to estimate \(\mu\) and \(\sigma\) in a manner consistent with the dynamics of interest. This reduces to the problem considered in Section 2, with the important distinction that the data used to fit the coefficients still come from the ground truth system \eqref{eq: multiscale SDE system}, whereas the target model \eqref{eq: effective SDE} is only an approximation. In statistical terminology, this may be viewed as a form of model misspecification in which the observed data are corrupted in a highly structured way. A common strategy is to derive estimators \(\hat{\mu}(\bar{x}_{\mathcal{D}_{\delta,T}})\) and \(\hat{\sigma}(\bar{x}_{\mathcal{D}_{\delta,T}})\) from a discrete-time sample of \eqref{eq: effective SDE}, and then to plug in the observed data \(x^{\varepsilon}_{\mathcal{D}_{\delta,T}}\), in the hope that if \(\bar{X}\) is well approximated by \(X^{\varepsilon}\), then the corresponding estimators should also be close. However, this continuity property does not hold in general, as illustrated by the following example.

\begin{example}
\label{ex: physical brownian motion}
Consider the simple model of physical Brownian motion. Let \(X^{\varepsilon}_t\) denote the position of a particle of mass \(\varepsilon>0\) moving along the real line subject to random impulses. This quantity may be modelled as
\begin{equation}
\label{eq: physical brownian motion}
\begin{cases}
\mathrm{d}X^{\varepsilon}_t = \frac{1}{\sqrt{\varepsilon}}Y^{\varepsilon}_t\,\mathrm{d}t,\\[0.3em]
\mathrm{d}Y^{\varepsilon}_t = -\frac{1}{\varepsilon}Y^{\varepsilon}_t\,\mathrm{d}t + \frac{\sigma}{\sqrt{\varepsilon}}\mathrm{d}W_t.
\end{cases}
\end{equation}
A straightforward calculation yields the following strong homogenisation estimate:
\[
\sup_{t\in[0,T]}\|X^{\varepsilon}(t)-\sigma W(t)\|_{L^2}\lesssim \sqrt{\varepsilon}.
\]
To use this model in practice, \(\sigma\) must be estimated, and the most natural estimator is the quadratic variation
\begin{equation}
\hat{\sigma}_{\varepsilon,\delta}^2
=
\frac{1}{T}\sum_{i=1}^N \bigl(X^{\varepsilon}_{i\delta}-X^{\varepsilon}_{(i-1)\delta}\bigr)^2.
\end{equation}
For a fixed value of \(\varepsilon>0\), \(\hat{\sigma}_{\varepsilon,\delta}^2\) is clearly inconsistent, since it converges to \(0\) as \(\delta\to 0\), as expected because \(X^{\varepsilon}\) has finite variation. More subtly, even if \(\delta=\delta(\varepsilon)\to 0\) as \(\varepsilon\to 0\), the estimator may still fail. Indeed,
\[
\mathbb{E}\left[\hat{\sigma}_{\delta,\varepsilon}^2\right]
=
\sigma^2\left\{1+\frac{\varepsilon}{\delta}\left(e^{-\delta/\varepsilon}-1\right)\right\},
\]
which converges to \(\sigma^2\) if \(\varepsilon/\delta\to 0\), but to \(0\) if \(\varepsilon/\delta\to \infty\).
\end{example}

This statistical inference problem for slow/fast systems of standard SDEs driven by Brownian motion has been solved in reasonably wide generality; see \cite{Pavliotis07, Papavasiliou09}. The most common setting is that of a single trajectory of the slow component observed over a fixed time horizon \(T\). Accordingly, the data typically consist of discrete-time samples \(x^{\varepsilon}_{\mathcal{D}_{\delta,T}}\) obtained from a single realisation of \(X^{\varepsilon}\). As the previous example suggests, the manner in which these samples are obtained from the trajectory determines the validity of the estimation procedure. Broadly speaking, the observation rate \(\delta\), namely the mesh size of the grid \(\mathcal{D}_{\delta,T}\) on which a trajectory of \(X^{\varepsilon}\) is sampled over the time horizon \([0,T]\), with \(\delta=T/N\), must be coupled with \(\varepsilon\) in such a way that
\[
\delta(\varepsilon)\xrightarrow[\varepsilon\to 0]{}0
\]
and
\[
\frac{\varepsilon}{f(\delta)}\to 0
\]
for an appropriate function \(f\). Although the choice of \(f\) is problem-specific, the identity \(f(\delta)=\delta\) often suffices. This second requirement, usually referred to as a \textit{subsampling condition}, slows down the observation rate in the sense that, for a fixed value of \(\varepsilon>0\), at most one observation of a trajectory \(X^{\varepsilon}(\omega)\) is taken in each cell of size \(\varepsilon\). Thus, under the multiscale assumption, the data should not be sampled on an arbitrarily fine grid.

Averaging and homogenisation results for standard SDEs have been known for several decades, with the first results in these directions appearing in the late 1960s. More recently, analogous questions have attracted increasing attention in the fractional setting, where the Brownian driving noise in \eqref{eq: multiscale SDE system} is replaced by fractional processes. This more general formulation introduces considerable flexibility, allowing effects such as long-range dependence or rougher pathwise behaviour than in the Brownian case to be modelled. The main example is fractional Brownian motion. This greater flexibility comes at the cost of losing two properties that are central to the analysis of both the effective dynamics and the associated inferential procedures, namely the Markov and semimartingale properties of the solutions to \eqref{eq: multiscale SDE system} and \eqref{eq: effective SDE}. In their absence, many of the standard techniques used to establish convergence for these problems break down, and new tools are required. The remainder of this section surveys two instances of this problem.
\subsection{Physical Fractional Brownian Motion model}

The model considered in this section is the following slow/fast system, which generalises \eqref{eq: physical brownian motion} by allowing for a much more flexible covariance structure. Studied in \cite{Alonso25}, it may be written as
\begin{equation}
\label{eq: physical fractional brownian motion}
\begin{cases}
\mathrm{d}X^{\varepsilon}_t = \varepsilon^{H-1}Y^{\varepsilon}_t\,\mathrm{d}t,\\[0.3em]
\mathrm{d}Y^{\varepsilon}_t = -\frac{1}{\varepsilon}Y^{\varepsilon}_t\,\mathrm{d}t + \frac{\sigma}{\varepsilon^H}\mathrm{d}B_t^H.
\end{cases}
\end{equation}
Known as the fractional kinetic, or physical, fractional Brownian motion, this model provides a physical description of particles subject to random forcing with memory. Just as standard kinetic Brownian motion converges to standard Brownian motion in the small-mass limit, one has here
\begin{equation}
\label{eq: strong homogenization fbm}
\|X^{\varepsilon}_t-\sigma B_t^H\|_{L^p}\lesssim \varepsilon^H
\end{equation}
for any \(p>1\) \cite{Gehringer20}.

\subsubsection{Fractional calculus for fractional Brownian motion}

The basic elements of fractional calculus needed in the sequel are now introduced. For a more detailed exposition, see \cite{Samko93}.

\begin{definition}
The fractional integrals of order \(\gamma\) are defined by
\begin{equation}
\label{eq: fractional integral}
\mathcal{I}_{\pm}^\gamma h(t)
=
\frac{1}{\Gamma(\gamma)} \int_{\mathbb{R}}(t-s)_{\pm}^{\gamma-1} h(s)\,\mathrm{d}s,
\end{equation}
whereas the Marchaud fractional derivatives of order \(\gamma\) are defined by
\begin{equation}
\label{eq: fractional derivative}
\mathcal{D}_{\pm}^\gamma h(t)
=
\frac{\gamma}{\Gamma(1-\gamma)} \int_0^{\infty} \frac{h(t)-h(t \mp s)}{s^{\gamma+1}}\,\mathrm{d}s.
\end{equation}
\end{definition}

\begin{remark}
In the literature, the Marchaud fractional derivative is typically defined as the \(L^p\)-limit obtained by truncating the integral on the right-hand side of \eqref{eq: fractional derivative} near the lower end of the integration domain, so as to include functions for which the integral converges only conditionally. This issue does not arise for the classes of functions considered here, and the technicality is therefore omitted for clarity.
\end{remark}

These operators are introduced because they provide a way to define integration with respect to fractional Brownian motion in the Wiener sense via an isometry. More precisely, let \(B_t^H\) be a fractional Brownian motion with Hurst parameter \(H\in(0,1)\). Then the Mandelbrot--Van Ness representation may be written as
\begin{equation}
\label{eq: mvn representation}
B_t^H
=
\int_{-\infty}^t C_H\bigl((t-s)^{H-1/2}-(-s)_+^{H-1/2}\bigr)\,\mathrm{d}W_s
=
\int_{-\infty}^t T_H\mathbbm{1}_{[0,t]}(s)\,\mathrm{d}W_s,
\end{equation}
where the operator \(T_H\) is defined by
\begin{equation}
\label{eq: KH kernel}
T_H =
\begin{cases}
\mathcal{I}^{H-1/2}_{-}, & \quad H>\frac{1}{2},\\[0.3em]
\mathcal{D}^{1/2-H}_{-}, & \quad H<\frac{1}{2}.
\end{cases}
\end{equation}
This isometry relation allows the Wiener integral with respect to \(B_t^H\), understood as the \(L^2\)-limit of Riemann sums against fractional Brownian motion, to be defined by
\begin{equation}
\label{eq: Wiener integral fBm}
\int_{\mathbb{R}} f(s)\,\mathrm{d}B_s^H
:=
\int_{\mathbb{R}} (T_H f)(s)\,\mathrm{d}W_s.
\end{equation}
In particular, it yields a form of It\^o isometry for fractional Brownian motion:
\begin{equation}
\mathbb{E}\left[\int_{\mathbb{R}} f(s)\,\mathrm{d}B_s^H \int_{\mathbb{R}} g(s)\,\mathrm{d}B_s^H \right]
=
\langle T_H f, T_H g \rangle_{L^2(\mathbb{R})},
\end{equation}
which follows directly from the standard It\^o isometry for the Wiener process. Two further properties of the fractional operators will play a key role in what follows.

\begin{proposition}[Invertibility of the fractional integral, Theorem 6.1, p.~125 in \cite{Samko93}]
Let \(f(x)=\mathcal{I}^{\alpha}_{\pm}\phi(x)\), where \(\phi\in L^p(\mathbb{R})\) and \(1\leq p<1/\alpha\). Then
\begin{equation}
\phi(x)=\mathcal{D}^{\alpha}_{\pm}f(x).
\end{equation}
\end{proposition}

\begin{proposition}[Adjointness of the derivatives, Corollary 2, p.~129 in \cite{Samko93}]
The fractional integration by parts formula
\begin{equation}
\int_{\mathbb{R}} f(x)\mathcal{D}^{\alpha}_{+}g(x)\,\mathrm{d}x
=
\int_{\mathbb{R}} \mathcal{D}^{\alpha}_{-}f(x)g(x)\,\mathrm{d}x
\end{equation}
is valid under the assumptions that \(\mathcal{D}^{\alpha}_{+}g\in L^r(\mathbb{R})\), \(\mathcal{D}^{\alpha}_{-}f\in L^p(\mathbb{R})\), \(f\in L^s(\mathbb{R})\), and \(g\in L^t(\mathbb{R})\), where \(p>1\), \(r>1\),
\[
\frac{1}{p}+\frac{1}{r}=1+\alpha,
\qquad
\frac{1}{s}=\frac{1}{p}-\alpha,
\qquad
\frac{1}{t}=\frac{1}{r}-\alpha.
\]
\end{proposition}

\subsubsection{Fractional Ornstein--Uhlenbeck process}

The fractional Ornstein--Uhlenbeck process was already introduced in Section 2 as a model for statistical inference. In the present context, the stationary version is needed, together with its explicit representation and covariance structure. More precisely, the (stationary) fractional Ornstein--Uhlenbeck process is the unique stationary solution to the Langevin equation
\begin{equation}
\mathrm{d}Y_t = -\lambda Y_t\,\mathrm{d}t + \beta\,\mathrm{d}B_t^H,
\end{equation}
and it admits the closed-form representation
\begin{equation}
Y_t = \beta \int_{-\infty}^t e^{-(t-s)\lambda}\,\mathrm{d}B_s^H.
\end{equation}
The fractional Ornstein--Uhlenbeck process of interest here is obtained by setting \(\lambda=1/\varepsilon\) and \(\beta=\sigma/\varepsilon^H\), which yields the rescaled process
\begin{equation}
\mathrm{d}Y^{\varepsilon}_t
=
-\frac{1}{\varepsilon}Y^{\varepsilon}_t\,\mathrm{d}t
+
\frac{\sigma}{\varepsilon^H}\,\mathrm{d}B_t^H.
\end{equation}
It follows that \(Y_t^{\varepsilon}=Y_{t/\varepsilon}^1\) in law, and the process is stationary and ergodic with invariant measure \(\mathcal{N}(0,\sigma_H^2)\), where \(\sigma_H^2=\sigma^2 c_H/2\). This explains the relevance of this choice in the models of interest, since it captures the effect of a random environment evolving on a much faster time scale than the quantity of interest. Another important property is its covariance structure, for which the following result holds.

\begin{proposition}[Covariance structure of the fOU process, Theorem 2.3 in \cite{Cheridito03}]
Let \(H\in\left(0,\frac{1}{2}\right)\cup\left(\frac{1}{2},1\right)\) and \(N=1,2,\ldots\). Then, for sufficiently small \(\varepsilon>0\),
\begin{equation}
\mathbb{E}\bigl[Y_t^{\varepsilon}Y_{t+s}^{\varepsilon}\bigr]
=
\frac{1}{2}\sigma^2
\sum_{n=1}^N
\left(\prod_{k=0}^{2n-1}(2H-k)\right)
(s/\varepsilon)^{2H-2n}
+
O\bigl((s/\varepsilon)^{2H-2N-2}\bigr).
\end{equation}
\end{proposition}

\subsection{Parameter estimation for the effective dynamics}

Having introduced the model and the homogenisation result, we now turn to the problem of estimating the parameters of interest from observations. This problem may be formulated as an instance of model misspecification within the inference framework considered in Section 2. Let \(X_t^{\varepsilon}\) be the slow component of the multiscale system \eqref{eq: physical fractional brownian motion}, observed at discrete times \(x_{\mathcal{D}_{\delta,T}}^{\varepsilon}\), assumed to be equispaced on a partition \(\mathcal{D}_{\delta,T}\) of a finite interval \([0,T]\) and sampled at a prescribed rate \(\delta\). The aim is to use these data to estimate \(\sigma\) and \(H\), exploiting the fact that for small \(\varepsilon\) the process \(X_t^{\varepsilon}\) is close to \(\sigma B_t^H\). The strategy is to compute robust statistics for the limiting process, following the approach of Section 2, and then to plug in the observed data. However, estimators of this type rely on high-frequency observations of the limiting process, which is precisely the regime in which the fast perturbations have a non-negligible effect, as in Example \ref{ex: physical brownian motion}. The main challenge is therefore to show that these estimators retain consistency and asymptotic normality when applied to data generated by the multiscale system, provided that the observations are collected at a suitable rate.

To estimate \(\sigma\), one uses the maximum likelihood estimator in sample space for discretised fractional Brownian motion, with the observed data \(x_{\mathcal{D}_{\delta,T}}^{\varepsilon}\) from \eqref{eq: physical fractional brownian motion} substituted into it:
\begin{equation}
\label{eq: estimator sigma definition}
\hat{\sigma}_{\varepsilon,\delta}^2
=
\frac{1}{N}
\left\langle
X_{\delta,T}^{\varepsilon},
(\Sigma_{\delta,T}^H)^{-1}X_{\delta,T}^{\varepsilon}
\right\rangle,
\qquad N=\frac{T}{\delta},
\end{equation}
where \(X_{\delta,T}^{\varepsilon}\) denotes the vector of increments of the observed process \(X^\varepsilon\), namely
\[
(X_{\delta,T}^{\varepsilon})^T
=
(x_\delta-x_0,\dots,x_{N\delta}-x_{(N-1)\delta}),
\]
and \(\Sigma_{\delta,T}^H\) denotes the covariance matrix of the increment vector \(\Delta B_{\delta,T}^H\) of a fractional Brownian motion observed on \(\mathcal{D}_{\delta,T}\). For \(H\), the estimator considered is
\begin{equation}
\label{eq: estimator H definition}
\hat{H}_{\varepsilon,\delta}
=
\frac{1}{2}
-
\frac{1}{2\log 2}
\log\left(
\frac{\sum_{k=2}^{2N}\left(\Delta_{k,\delta/2}^{(2)}X^{\varepsilon}\right)^2}
{\sum_{k=2}^{N}\left(\Delta_{k,\delta}^{(2)}X^{\varepsilon}\right)^2}
\right),
\end{equation}
where
\[
\Delta_{k,\delta}^{(2)}X^\varepsilon
=
X_{k\delta}^{\varepsilon}
-
2X_{(k-1)\delta}^{\varepsilon}
+
X_{(k-2)\delta}^{\varepsilon}
\]
denotes the second-order increment of the process.

Before analysing these estimators directly, we derive the key tool needed to control the error introduced by using multiscale data in place of the limiting dynamics in the MLE.

\subsubsection{Spectral analysis of the inverse covariance matrix}

The computation of the MLE involves the inverse of a covariance matrix. In addition to the computational burden this may impose, which lies beyond the scope of the present discussion, although its Toeplitz structure can be exploited to make the computation more efficient, it also introduces a potential source of instability that must be controlled. This is particularly relevant in the present setting, where an approximation of the limiting process is used, thereby introducing an error that may in principle be amplified by the action of the inverse matrix. To control this effect, it is necessary to understand the spectral properties of the covariance matrix \(\Sigma_{\delta,T}^H\). In particular, the aim is to control
\[
\|(\Sigma_{\delta,T}^H)^{-1}\|_2,
\]
where \(\|\cdot\|_2\) denotes the operator norm induced by the Euclidean norm on \(\mathbb{R}^N\), that is, the spectral norm. The following result is used.

\begin{lemma}[{\cite{Alonso25}}]
\label{orderlemma}
Let \(\Sigma_{\delta,T}^H\) be the covariance matrix of \(\Delta B_{\delta,T}^H\). Then
\[
\|(\Sigma_{\delta,T}^H)^{-1}\|_2 \leq C N^{\beta},
\]
where \(N=T/\delta\) and \(\beta=\max\{1,2H\}\).
\end{lemma}

The spectral norm of the inverse covariance matrix is bounded by the reciprocal of the smallest eigenvalue. This reformulates the problem as that of finding a lower bound for the quadratic form associated with \(\Sigma_{\delta,T}^H\), of the form
\[
u^T \Sigma_{\delta,T}^H u \geq C(N)\sum_{i=1}^N u_i^2
\]
for any eigenvector \(u\in\mathbb{R}^N\) associated with \(\Sigma_{\delta,T}^H\), since \(C(N)^{-1}\) is then automatically an upper bound for the desired quantity.

Moreover, the quadratic form associated with \(\Sigma_{\delta,T}^H\), that is, with the covariance matrix of fractional Gaussian noise, may be written as the second moment of a Wiener integral with respect to the underlying fractional Brownian motion:
\[
\begin{aligned}
u^T \Sigma_{\delta,T}^H u
&=
\sum_{i,j=0}^{N-1}u_i u_j (\Sigma_{\delta,T}^H)_{ij}
\\
&=
\sum_{i,j=0}^{N-1}u_i u_j
\mathbb{E}\left[
\bigl(B_{(i+1)\delta}^H-B_{i\delta}^H\bigr)
\bigl(B_{(j+1)\delta}^H-B_{j\delta}^H\bigr)
\right]
\\
&=
\mathbb{E}\left[
\left(
\sum_{i=0}^{N-1}u_i\bigl(B_{(i+1)\delta}^H-B_{i\delta}^H\bigr)
\right)^2
\right]
=
\mathbb{E}\left[
\left(
\int_0^T f(t)\,\mathrm{d}B_t^H
\right)^2
\right],
\end{aligned}
\]
where
\[
f(t)=\sum_{i=0}^{N-1}u_i\mathbf{1}_{[i\delta,(i+1)\delta)}(t)
\]
is a step function on \([0,T]\).

If \(\mathcal{H}\) denotes the completion of the space of step functions on \([0,T]\) under the inner product
\[
\langle \mathbbm{1}_{[0,t]},\mathbbm{1}_{[0,s]} \rangle_{\mathcal{H}}
=
\mathbb{E}[B_t^H B_s^H],
\]
then finding a lower bound for the quadratic form associated with \(\Sigma_{\delta,T}^H\) is equivalent to finding a lower bound for the norm on this space, which is precisely the domain of the Wiener integral with respect to fractional Brownian motion. The proof relies on various representations of this space and of the norm \(\|\cdot\|_{\mathcal{H}}\), which differ substantially according to the regime of \(H\).

The case \(H<1/2\) is considered first. In this regime, the aim is to establish directly an estimate of the form
\[
\|f\|_{\mathcal{H}} \geq C\|f\|_{L^2([0,T])}.
\]
Indeed, for \(f\in \mathcal{H}\) and \(H<1/2\), one may write (see \cite{Bardina06}, Theorem 2.5)
\[
\|f\|_{\mathcal{H}}^2
=
\frac{1}{2}H(1-2H)
\iint_{\mathbb{R}^2}
\frac{(f(x)-f(y))^2}{|x-y|^{2-2H}}
\,\mathrm{d}x\,\mathrm{d}y.
\]
Moreover, for any compactly supported function \(f\) on the real line, it follows from \cite{Dinezza12}, Theorem 6.5, that
\[
\iint_{\mathbb{R}^2}
\frac{(f(x)-f(y))^2}{|x-y|^{2-2H}}
\,\mathrm{d}x\,\mathrm{d}y
\geq C\|f\|_{L^q([0,T])}
\]
for any \(q\in[2,1/H]\). Taking
\[
f=\sum_{i=0}^{N-1}u_i\mathbf{1}_{[i\delta,(i+1)\delta)}
\]
and combining the two identities above yields
\[
\|f\|_{\mathcal{H}}^2
\geq
C\|f\|_{L^2(\mathbb{R})}^2
=
C\|f\|_{L^2([0,T])}^2
=
C\delta\sum_{i=0}^{N-1}u_i^2,
\]
and hence
\[
u^T \Sigma_{\delta,T}^H u \geq C\delta\sum_{i=0}^{N-1}u_i^2.
\]
It follows that, if \(H<1/2\),
\[
\|(\Sigma_{\delta,T}^H)^{-1}\|_2 \leq C\delta^{-1}.
\]

The case \(H>1/2\) requires a different argument. In this regime, the domain of the Wiener integral with respect to fractional Brownian motion can be identified with the space \(\mathcal{I}_{-}^{H-1/2}(L^2(\mathbb{R}))\), and the norm is given by
\begin{equation}
\|f\|_{\mathcal{H}}
=
\|\mathcal{I}_{-}^{H-1/2}f\|_{L^2(\mathbb{R})}
=
\|\mathcal{I}_{+}^{H-1/2}f\|_{L^2(\mathbb{R})}.
\end{equation}
Using the properties above, one obtains the following estimate.

\begin{proposition}[{\cite{Alonso25}}]
Let \(H>1/2\) and \(f\in \mathcal{E}_{[0,T]}\). Then
\[
\|f\|_{L^2(\mathbb{R})}^2
\leq
c_{H,T}\|f\|_{\mathcal{H}}
\left\|\mathcal{D}_{-}^{H-1/2}f\right\|_{L^2([0,T])}.
\]
\end{proposition}

\begin{proof}
Cauchy--Schwarz's inequality yields
\[
\begin{aligned}
\left|\langle f,f\rangle_{L^2(\mathbb{R})}\right|
&=
\left|
\left\langle
\mathcal{D}_{+}^{H-1/2}\mathcal{I}_{+}^{H-1/2}f,
f
\right\rangle_{L^2(\mathbb{R})}
\right|
\\
&=
\left|
\left\langle
\mathcal{I}_{+}^{H-1/2}f,
\mathcal{D}_{-}^{H-1/2}f
\right\rangle_{L^2([0,T])}
\right|
\\
&\leq
\left\|\mathcal{I}_{+}^{H-1/2}f\right\|_{L^2(\mathbb{R}^+)}
\left\|\mathcal{D}_{-}^{H-1/2}f\right\|_{L^2([0,T])}
\\
&=
c_{H,T}\|f\|_{\mathcal{H}}
\left\|\mathcal{D}_{-}^{H-1/2}f\right\|_{L^2([0,T])}.
\end{aligned}
\]
The fact that \(\mathcal{E}_{[0,T]}\subset L^p(\mathbb{R})\) for any \(1\leq p\leq \infty\) guarantees that the first two equalities hold. The change in domain from \(\mathbb{R}\) to \([0,T]\) in the second equality is due to the fact that \(\mathcal{I}_{+}^{\gamma}f\) is supported on \((0,+\infty)\) and \(\mathcal{D}_{-}^{\gamma}f\) on \((-\infty,T)\) when \(f\) is supported on \([0,T)\).
\end{proof}

Using this estimate, the desired quantity can be bounded by
\[
\|f\|_{\mathcal{H}}^2
\geq
\left(
\frac{\|f\|_{L^2(\mathbb{R}^+)}^2}
{\sqrt{c_{H,T}}\left\|\mathcal{D}_{-}^{H-1/2}f\right\|_{L^2(\mathbb{R}^+)}}
\right)^2.
\]
Moreover,
\begin{align*}
\|f\|_{L^2(\mathbb{R}^+)}^2
&=
\int_0^T
\left(
\sum_{i=0}^{N-1}u_i\mathbf{1}_{[i\delta,(i+1)\delta)}(s)
\right)^2
\,\mathrm{d}s
\\
&=
\int_0^T
\sum_{i=0}^{N-1}u_i^2\mathbf{1}_{[i\delta,(i+1)\delta)}(s)
\,\mathrm{d}s
\\
&=
\delta\sum_{i=0}^{N-1}u_i^2,
\end{align*}
and
\[
\left\|\mathcal{D}_{-}^{\gamma}f\right\|_{L^2(\mathbb{R}^+)}^2
\leq
C\delta^{1-2\gamma}\sum_{i=0}^{N-1}u_i^2
\]
for any \(\gamma\in(0,1/2)\). The proof of this estimate relies on a lengthy calculation, for which the reader is referred to the supplementary material in \cite{Alonso25}. It follows that
\[
u^T \Sigma_{\delta,T}^H u
\geq
\left(
\frac{\delta\sum_{i=0}^{N-1}u_i^2}
{C\delta^{1/2-\gamma}\left(\sum_{i=0}^{N-1}u_i^2\right)^{1/2}}
\right)^2
=
C\delta^{1+2\gamma}\sum_{i=0}^{N-1}u_i^2,
\]
and therefore, for \(\gamma=H-1/2\) and \(H\in(1/2,1)\),
\[
\|(\Sigma_{\delta,T}^H)^{-1}\|_2
\leq
C\delta^{-1-2\gamma}
=
C\delta^{-2H}
=
C\left(\frac{N}{T}\right)^{2H}.
\]
Since \(T\) is fixed, this is of order \(N^{2H}\).

\subsubsection{Failure of the estimator without subsampling}

The following result shows that the estimator fails without an appropriate subsampling condition.

\begin{theorem}[{\cite{Alonso25}}]
\label{noconvergence}
Let \(X^{\varepsilon}\) be the solution of \eqref{eq: physical fractional brownian motion}, and let \(T=N\delta\) with \(0<H<1\). For any fixed \(\varepsilon>0\), it holds that
\begin{equation}
\label{unbiasedno}
\lim_{\delta\to 0}\hat{\sigma}_{\varepsilon,\delta}^{2} = 0
\qquad \text{a.s.},
\end{equation}
where \(\hat{\sigma}_{\varepsilon,\delta}\) is defined in \eqref{eq: estimator sigma definition}. Moreover, if \(\delta(\varepsilon)>0\) is such that \(\delta=\varepsilon^{\alpha}\) for some \(\alpha>0\) with \(\alpha>\max\{1,2(1-H)\}\), then
\begin{equation}
\label{unbiasedia}
\mathbb{E}[\hat{\sigma}_{\varepsilon,\delta}^{2}]
\xrightarrow[\varepsilon\to 0^+]{}
0.
\end{equation}
\end{theorem}

\begin{proof}
The proof is standard, and the reader is referred to the original manuscript \cite{Alonso25}.
\end{proof}

\subsubsection{Consistency of the estimators under appropriate subsampling}

We now turn to the consistency of the estimators under suitable subsampling.

\begin{theorem}[{\cite{Alonso25}}]
\label{maintheorem}
Let \(\hat{\sigma}_{\varepsilon,\delta}^2\) be the estimator defined in \eqref{eq: estimator sigma definition}, constructed from the solution \(X_t^{\varepsilon}\) of \eqref{eq: physical fractional brownian motion}, observed at sampling rate \(\delta=\varepsilon^{\alpha}\), where \(\varepsilon>0\) is the scale-separation parameter in \eqref{eq: physical fractional brownian motion}. For any \(0<H<1\) and \(0<\alpha<\min\left\{1,\frac{H}{1-H}\right\}\), it holds that
\[
\hat{\sigma}_{\varepsilon,\delta}^{2}
\xrightarrow[\varepsilon\to 0^+]{}
\sigma^2
\qquad
\text{in }L^2.
\]
\end{theorem}

\begin{proof}
The estimator is first decomposed and the \(L^2\)-norm applied:
\begin{align}
\label{eq: estimator decomposition}
\left\|\hat{\sigma}_{\varepsilon,\delta}^{2}-\sigma^2\right\|_{L^2}
\leq\,&
\frac{1}{N}
\left\|
\bigl(X_{\delta,T}^{\varepsilon}-\sigma B_{\delta,T}^{H}\bigr)^T
(\Sigma_{\delta,T}^H)^{-1}
\bigl(X_{\delta,T}^{\varepsilon}-\sigma B_{\delta,T}^{H}\bigr)
\right\|_{L^2}
\nonumber\\
&+
\frac{2\sigma}{N}
\left\|
\bigl(B_{\delta,T}^{H}\bigr)^T
(\Sigma_{\delta,T}^H)^{-1}
\bigl(X_{\delta,T}^{\varepsilon}-\sigma B_{\delta,T}^{H}\bigr)
\right\|_{L^2}
\nonumber\\
&+
\left\|
\frac{\sigma^2}{N}
\bigl(B_{\delta,T}^{H}\bigr)^T
(\Sigma_{\delta,T}^H)^{-1}
\bigl(B_{\delta,T}^{H}\bigr)
-
\sigma^2
\right\|_{L^2}.
\end{align}
The difference may be written as
\[
X_{\delta,T}^{\varepsilon}-\sigma B_{\delta,T}^{H}
=
\sigma_1\varepsilon^{H}Y_{\delta,T}^{\varepsilon},
\]
where \(Y_{\delta,T}^{\varepsilon}\) denotes the vector of increments of \(Y^\varepsilon\) (see \cite{Gehringer20}).

For the first term,
\begin{align}
\label{asbound1}
\bigl(X_{\delta,T}^{\varepsilon}-\sigma B_{\delta,T}^{H}\bigr)^T
(\Sigma_{\delta,T}^H)^{-1}
\bigl(X_{\delta,T}^{\varepsilon}-\sigma B_{\delta,T}^{H}\bigr)
&\leq
\|X_{\delta,T}^{\varepsilon}-\sigma B_{\delta,T}^{H}\|_2^2
\|(\Sigma_{\delta,T}^H)^{-1}\|_2
\nonumber\\
&=
\|\sigma_1\varepsilon^{H}Y_{\delta,T}^{\varepsilon}\|_2^2
\|(\Sigma_{\delta,T}^H)^{-1}\|_2.
\end{align}
Since \(\sigma_1^2\|(\Sigma_{\delta,T}^H)^{-1}\|_2\) is deterministic, it suffices to control
\begin{align}
\label{roughboundCS1}
\mathbb{E}\left[\|\varepsilon^{H}Y_{\delta,T}^{\varepsilon}\|_2^4\right]
&=
\varepsilon^{4H}
\mathbb{E}\left[
\left(\sum_{i=0}^{N-1}(Y_{i\delta,(i+1)\delta}^{\varepsilon})^2\right)
\left(\sum_{j=0}^{N-1}(Y_{j\delta,(j+1)\delta}^{\varepsilon})^2\right)
\right]
\nonumber\\
&=
\varepsilon^{4H}
\sum_{i,j=0}^{N-1}
\mathbb{E}\left[
(Y_{i\delta,(i+1)\delta}^{\varepsilon})^2
(Y_{j\delta,(j+1)\delta}^{\varepsilon})^2
\right]
\nonumber\\
&\leq
\varepsilon^{4H}
\sum_{i,j=0}^{N-1}
\mathbb{E}\left[(Y_{i\delta,(i+1)\delta}^{\varepsilon})^4\right]^{1/2}
\mathbb{E}\left[(Y_{j\delta,(j+1)\delta}^{\varepsilon})^4\right]^{1/2}
\nonumber\\
&=
\varepsilon^{4H}N^2
\mathbb{E}\left[(Y_{0,\delta}^{\varepsilon})^4\right]
=
3\varepsilon^{4H}N^2
\mathbb{E}\left[(Y_{0,\delta}^{\varepsilon})^2\right]^2,
\end{align}
where Cauchy--Schwarz's inequality and stationarity of the fractional Ornstein--Uhlenbeck process \(Y^{\varepsilon}\) have been used. It follows that
\[
\left\|
\|\varepsilon^{H}Y_{\delta,T}^{\varepsilon}\|_2^2
\right\|_{L^2}
=
\left(
\mathbb{E}\left[\|\varepsilon^{H}Y_{\delta,T}^{\varepsilon}\|_2^4\right]
\right)^{1/2}
\lesssim
\varepsilon^{2H}N.
\]
Together with Lemma \ref{orderlemma}, this yields
\begin{equation}
\label{term1}
\frac{1}{N}
\left\|
\bigl(X_{\delta,T}^{\varepsilon}-\sigma B_{\delta,T}^{H}\bigr)^T
(\Sigma_{\delta,T}^H)^{-1}
\bigl(X_{\delta,T}^{\varepsilon}-\sigma B_{\delta,T}^{H}\bigr)
\right\|_{L^2}
\lesssim
\varepsilon^{2H}/\delta^{\beta}.
\end{equation}

A similar argument applies to the second term in \eqref{eq: estimator decomposition}. First,
\begin{equation}
\label{asbound2}
\left\|
\bigl(\sigma B_{\delta,T}^{H}\bigr)^T
(\Sigma_{\delta,T}^H)^{-1}
\bigl(X_{\delta,T}^{\varepsilon}-\sigma B_{\delta,T}^{H}\bigr)
\right\|_{L^2}
\leq
\|\sigma B_{\delta,T}^{H}\|_2
\|\sigma_1\varepsilon^{H}Y_{\delta,T}^{\varepsilon}\|_2
\|(\Sigma_{\delta,T}^H)^{-1}\|_2.
\end{equation}
It is therefore enough to control
\[
\left\|
\|\sigma_1\varepsilon^{H}Y_{\delta,T}^{\varepsilon}\|_2
\|\sigma B_{\delta,T}^{H}\|_2
\right\|_{L^2}.
\]
One has
\begin{align}
\label{roughboundCS2}
\mathbb{E}\left[
\|\sigma_1\varepsilon^{H}Y_{\delta,T}^{\varepsilon}\|_2^2
\|\sigma B_{\delta,T}^{H}\|_2^2
\right]
&=
\varepsilon^{2H}\sigma^2\sigma_1^2
\sum_{i,j=0}^{N-1}
\mathbb{E}\left[
(Y_{i\delta,(i+1)\delta}^{\varepsilon})^2
(B_{j\delta,(j+1)\delta}^{H})^2
\right]
\nonumber\\
&\leq
\varepsilon^{2H}\sigma^2\sigma_1^2
\sum_{i,j=0}^{N-1}
\mathbb{E}\left[(Y_{i\delta,(i+1)\delta}^{\varepsilon})^4\right]^{1/2}
\mathbb{E}\left[(B_{j\delta,(j+1)\delta}^{H})^4\right]^{1/2}
\nonumber\\
&=
\varepsilon^{2H}\sigma^2\sigma_1^2 N^2
\mathbb{E}\left[(Y_{0,\delta}^{\varepsilon})^4\right]^{1/2}
\mathbb{E}\left[(B_{0,\delta}^{H})^4\right]^{1/2}
\nonumber\\
&=
3\varepsilon^{2H}\sigma^2\sigma_1^2 N^2
\mathbb{E}\left[(Y_{0,\delta}^{\varepsilon})^2\right]
\mathbb{E}\left[(B_{0,\delta}^{H})^2\right]
=
3\varepsilon^{2H}\sigma^2\sigma_1^2 N^2\delta^{2H}.
\end{align}
Hence
\[
\left\|
\|\sigma_1\varepsilon^{H}Y_{\delta,T}^{\varepsilon}\|_2
\|\sigma B_{\delta,T}^{H}\|_2
\right\|_{L^2}
\lesssim
N\varepsilon^{H}\delta^{H},
\]
which, together with Lemma \ref{orderlemma}, gives
\[
\frac{2\sigma}{N}
\left\|
\bigl(B_{\delta,T}^{H}\bigr)^T
(\Sigma_{\delta,T}^H)^{-1}
\bigl(X_{\delta,T}^{\varepsilon}-\sigma B_{\delta,T}^{H}\bigr)
\right\|_{L^2}
\lesssim
\varepsilon^{H}/\delta^{\beta-H}.
\]

Finally, the third term in \eqref{eq: estimator decomposition} is controlled by a law-of-large-numbers-type argument. Since \(\Sigma_{\delta,T}^H\) is the covariance matrix of \(B_{\delta,T}^{H}\), it follows that
\[
\bigl(B_{\delta,T}^{H}\bigr)^T (\Sigma_{\delta,T}^H)^{-1} B_{\delta,T}^{H}
\sim \chi_N^2.
\]
Therefore,
\begin{align*}
\left\|
\frac{\sigma^2}{N}\bigl(B_{\delta,T}^{H}\bigr)^T (\Sigma_{\delta,T}^H)^{-1} B_{\delta,T}^{H}-\sigma^2
\right\|_{L^2}
&=
\sigma^2
\left(
\frac{1}{N^2}\mathbb{E}\bigl[(\chi_N^2)^2\bigr]
-
2\frac{1}{N}\mathbb{E}[\chi_N^2]
+
1
\right)^{1/2}
\\
&=
\sigma^2\left(\frac{2}{N}\right)^{1/2}
=
\sigma^2\left(\frac{2\delta}{T}\right)^{1/2}.
\end{align*}

Combining the above bounds, one obtains, for \(H>1/2\),
\begin{equation}
\label{eq: order large}
\left\|\hat{\sigma}_{\varepsilon,\delta}^2-\sigma^2\right\|_{L^2}
\lesssim
\left(\frac{\varepsilon}{\delta}\right)^{2H}
+
\left(\frac{\varepsilon}{\delta}\right)^{H}
+
\delta^{1/2},
\end{equation}
whereas, for \(H<1/2\),
\begin{equation}
\label{eq: order small}
\left\|\hat{\sigma}_{\varepsilon,\delta}^2-\sigma^2\right\|_{L^2}
\lesssim
\frac{\varepsilon^{2H}}{\delta}
+
\frac{\varepsilon^{H}}{\delta^{1-H}}
+
\delta^{1/2}.
\end{equation}
Therefore, if \(H>1/2\), convergence is ensured provided that \(\varepsilon/\delta\to 0\), whereas if \(H<1/2\), the requirement is \(\varepsilon^H/\delta^{1-H}\to 0\). In particular, if \(\delta=\varepsilon^{\alpha}\) for some \(\alpha>0\), convergence holds for
\[
0<\alpha<\min\left\{1,\frac{H}{1-H}\right\}.
\]
\end{proof}

The Hurst parameter may be estimated as follows.

\begin{proposition}[{\cite{Alonso25}}]
Let \(\hat{H}_{\delta,\varepsilon}\) be defined by
\begin{equation}
\label{eq: estimator of H}
\hat{H}_{\delta,\varepsilon}
=
\frac{1}{2}
-
\frac{1}{2\ln 2}
\ln\left(
\frac{\sum_{k=2}^{2N}\left(\Delta_{k,\delta/2}^{(2)}X^{\varepsilon}\right)^2}
{\sum_{k=2}^{N}\left(\Delta_{k,\delta}^{(2)}X^{\varepsilon}\right)^2}
\right),
\end{equation}
where
\[
\Delta_{k,\delta}^{(2)}X
=
X_{k\delta}-2X_{(k-1)\delta}+X_{(k-2)\delta}
\]
denotes the second-order difference. Assume that \(\delta(\varepsilon)\to0\) and \(\varepsilon/\delta\to 0\) as \(\varepsilon\to 0\). Then
\[
\hat{H}_{\delta,\varepsilon}
\xrightarrow[\varepsilon\to0^+]{}
H
\qquad
\text{in probability.}
\]
\end{proposition}

\begin{proof}
The first step is to derive the estimate
\begin{equation}
\label{eq: estimate on difference of sums of squares second order differences}
N^{2H-1}
\left|
\sum_{k=2}^{N}\left(\Delta_{k,\delta}^{(2)}X^{\varepsilon}\right)^2
-
\sum_{k=2}^{N}\left(\sigma\Delta_{k,\delta}^{(2)}B^H\right)^2
\right|
\lesssim
\left(\frac{\varepsilon}{\delta}\right)^{2H}
+
\left(\frac{\varepsilon}{\delta}\right)^H
\end{equation}
in \(L^2(\Omega)\), using the identity
\begin{align*}
\left(\Delta_{k,\delta}^{(2)}X^{\varepsilon}\right)^2
-
\left(\sigma\Delta_{k,\delta}^{(2)}B^H\right)^2
&=
\left(\Delta_{k,\delta}^{(2)}X^{\varepsilon}-\sigma\Delta_{k,\delta}^{(2)}B^H\right)^2
\\
&\quad
+
2\left(\Delta_{k,\delta}^{(2)}X^{\varepsilon}-\sigma\Delta_{k,\delta}^{(2)}B^H\right)
\left(\sigma\Delta_{k,\delta}^{(2)}B^H\right).
\end{align*}
Recalling that
\[
X_{s,t}^{\varepsilon}-\sigma B_{s,t}^H
=
\varepsilon^H\sigma_1 Y_{s,t}^{\varepsilon}
\]
for any \(s,t\), it follows that
\[
\Delta_{k,\delta}^{(2)}X^{\varepsilon}
-
\sigma\Delta_{k,\delta}^{(2)}B^H
=
\varepsilon^H\sigma_1\Delta_{k,\delta}^{(2)}Y^{\varepsilon}.
\]
Under the stationarity assumption on \(Y^{\varepsilon}\), a direct application of Cauchy--Schwarz's inequality gives
\begin{align*}
\varepsilon^{2H}\sigma_1^2
\left\|
\bigl(\Delta_{k,\delta}^{(2)}Y^{\varepsilon}\bigr)^2
\right\|_{L^2}
&\lesssim
\varepsilon^{2H},
\\
\varepsilon^{H}\sigma\sigma_1
\left\|
\bigl(\Delta_{k,\delta}^{(2)}Y^{\varepsilon}\bigr)
\bigl(\Delta_{k,\delta}^{(2)}B^{H}\bigr)
\right\|_{L^2}
&\lesssim
\varepsilon^{H}\delta^{H}.
\end{align*}
With these estimates, the result follows by the same argument as in \cite{Bourguin21a}.
\end{proof}

\subsubsection{Asymptotic normality of the estimator}

A central limit theorem type result follows directly from the previous consistency result and the classical CLT. The only point requiring additional care is that the estimator is only asymptotically unbiased and is not exact for finite samples. It is therefore necessary to verify that the renormalisation required to observe non-trivial fluctuations around the true parameter value preserves this property. The error was previously decomposed as
\[
\hat{\sigma}_{\varepsilon,\delta}^{2}-\sigma^2
=
A_1^{\varepsilon,\delta}
+
A_2^{\varepsilon,\delta}
+
\sigma^2\left(
\frac{1}{N-1}(\Delta B_{\delta,T}^H)^T(\Sigma_{\delta,T}^H)^{-1}\Delta B_{\delta,T}^H-1
\right),
\]
where, in the case \(H>1/2\),
\[
\|A_1^{\varepsilon,\delta}\|_{L^2}\lesssim (\varepsilon/\delta)^{2H},
\qquad
\|A_2^{\varepsilon,\delta}\|_{L^2}\lesssim (\varepsilon/\delta)^{H},
\]
whereas, if \(H<1/2\),
\[
\|A_1^{\varepsilon,\delta}\|_{L^2}\lesssim \varepsilon^{2H}/\delta,
\qquad
\|A_2^{\varepsilon,\delta}\|_{L^2}\lesssim \varepsilon^H/\delta^{1-H}.
\]
The limiting fluctuation is
\begin{equation}
\label{eq: CLT}
\frac{\hat{\sigma}_{\varepsilon,\delta}^{2}-\sigma^2}{\sqrt{\delta}}
\xrightarrow[\varepsilon\to 0]{}
\mathcal{N}(0,2\sigma^2).
\end{equation}
Thus,
\[
\frac{\hat{\sigma}_{\varepsilon,\delta}^{2}-\sigma^2}{\sqrt{\delta}}
=
A_1^{\varepsilon,\delta}\delta^{-1/2}
+
A_2^{\varepsilon,\delta}\delta^{-1/2}
+
\frac{\sigma^2}{\sqrt{\delta}}(S_N-1),
\]
where
\[
S_N=\frac{1}{N-1}(\Delta B_{\delta,T}^H)^T(\Sigma_{\delta,T}^H)^{-1}\Delta B_{\delta,T}^H
\]
is distributed as \(\frac{1}{N-1}\sum_{i=1}^{N-1}\chi_i^2\), where \(\{\chi_i^2\}_{i=1}^{N-1}\) is a sequence of i.i.d. chi-squared random variables with one degree of freedom. It follows immediately that the last term converges in distribution to \(\mathcal{N}(0,2\sigma^2)\) as \(N\to\infty\). It remains to ensure that the error terms vanish in probability after rescaling. This is achieved by requiring \(\delta=\varepsilon^{\alpha}\) for some \(\alpha<\frac{H}{1/2+H}\) when \(H>1/2\), and \(\alpha<\frac{H}{3/2-H}\) when \(H<1/2\). This yields the following compact formulation.

\begin{theorem}[{\cite{Alonso25}}]
Let \(\hat{\sigma}_{\varepsilon,\delta}^2\) be the estimator defined in \eqref{eq: estimator sigma definition}, constructed from the solution \(X_t^{\varepsilon}\) of \eqref{eq: physical fractional brownian motion}, observed at sampling rate \(\delta=\varepsilon^{\alpha}\), where \(\varepsilon>0\) is the scale-separation parameter in \eqref{eq: physical fractional brownian motion}. For any \(0<H<1\) and \(0<\alpha<\min\left\{\frac{H}{1/2+H},\frac{H}{3/2-H}\right\}\), it holds that
\begin{equation}
\frac{\hat{\sigma}_{\varepsilon,\delta}^{2}-\sigma^2}{\sqrt{\delta}}
\xrightarrow[\varepsilon\to 0]{}
\mathcal{N}(0,2\sigma^2).
\end{equation}
\end{theorem}
\subsection{Slow/fast systems perturbed by fractional noise}

A related problem is studied in \cite{Bourguin21a}, where the slow variable of a slow/fast system is itself perturbed by fractional noise. In contrast to the setting of Section 3.3, where the limiting dynamics remain stochastic, the limiting model here is deterministic. This makes it natural to estimate the averaged drift by fitting observed trajectories directly to the limiting dynamics.

Consider the system \((X^{\eta,\varepsilon,\theta},Y^{\varepsilon})\in\mathcal{X}\times\mathcal{Y}=\mathbb{R}^m\times\mathbb{R}^{d-m}\) solving
\begin{equation}
\label{eq: perturbed fractional multiscale}
\begin{cases}
\mathrm{d}X_t^{\eta,\varepsilon,\theta}
=
c_{\theta}(X_t^{\eta,\varepsilon,\theta},Y_t^{\varepsilon})\,\mathrm{d}t
+
\sqrt{\eta}\,\sigma(Y_t^{\varepsilon})\,\mathrm{d}B_t^{H},
\qquad
X_0^{\eta,\varepsilon,\theta}=x_0\in\mathcal{X},
\\[0.4em]
\mathrm{d}Y_t^{\varepsilon}
=
\frac{1}{\varepsilon}f(Y_t^{\varepsilon})\,\mathrm{d}t
+
\frac{1}{\sqrt{\varepsilon}}\tau(Y_t^{\varepsilon})\,\mathrm{d}B_t,
\qquad
Y_0^{\varepsilon}=y_0\in\mathcal{Y},
\end{cases}
\end{equation}
where \(B_t^H\) is a fractional Brownian motion with Hurst parameter \(1/2<H<1\), and \(B_t\) is an independent standard Brownian motion. In the usual slow/fast notation, \(\varepsilon>0\) is the time-scale separation parameter, while \(\eta>0\) is an additional small parameter controlling the intensity of the fractional perturbation acting on the slow variable. The vector \(\theta\in\Theta\) parametrises the drift of the slow motion, where \(\Theta\) is an open, bounded, and convex subset of a Euclidean space.

The large-scale--small-noise limit in this setting resembles an averaging effect, and under suitable assumptions on the coefficients (see \cite{Bourguin21b}) one can prove the following strong averaging estimate:
\begin{equation}
\label{eq: strong averaging result}
\mathbb{E}\sup_{0\leq t\leq T}\left|X_t^{\eta,\varepsilon,\theta}-\bar{X}_t^{\theta}\right|^p
\leq
K(\varepsilon^{p/2}+\eta^{p/2}),
\end{equation}
where \(\bar{X}_t^{\theta}\) solves the ODE
\begin{equation}
\label{eq: averaged ODE}
\mathrm{d}\bar{X}_t^{\theta}
=
\bar{c}_{\theta}(\bar{X}_t^{\theta})\,\mathrm{d}t,
\qquad
\bar{X}_0^{\theta}=x_0.
\end{equation}
As before, the relevant question is how to model effectively the drift of the system on the appropriate time scale when the data come from \eqref{eq: perturbed fractional multiscale}. In this parametrised setting, this amounts to estimating \(\theta\). The deterministic nature of the limit, together with the strong averaging estimate, which quantifies the closeness of realised paths to the limit rather than merely convergence in law, allows for inference strategies that differ from those used in the physical fractional Brownian motion model.

The problem of estimating the averaged drift from discrete-time data \(x_{\mathcal{D}_{\delta,T}}^{\eta,\varepsilon,\theta}\) generated by \eqref{eq: perturbed fractional multiscale} is addressed in \cite{Bourguin21a}, where trajectory fitting estimators (TFEs) are introduced. A loss function is chosen to penalise the discrepancy between the observed values and the limiting trajectory corresponding to a given value of \(\theta\); for instance, the mean squared error
\begin{equation}
U(\theta;x_{\mathcal{D}_{\delta,T}}^{\eta,\varepsilon,\theta})
=
\sum_{k=1}^{N}\left|X_{t_k}^{\eta,\varepsilon,\theta}-\bar{X}_{t_k}^{\theta}\right|^2,
\qquad N=\frac{T}{\delta}.
\end{equation}
The corresponding TFE is then defined by
\begin{equation}
\hat{\theta}_{\mathrm{TFE}}(x_{\mathcal{D}_{\delta,T}}^{\eta,\varepsilon,\theta})
=
\arg\min_{\theta\in\Theta}
U(\theta;x_{\mathcal{D}_{\delta,T}}^{\eta,\varepsilon,\theta}).
\end{equation}
In \cite{Bourguin21a}, guarantees are established for this estimator in two different regimes. First, they consider the problem of estimating the averaged drift from a finite sample. Under suitable technical assumptions on the coefficients, consistency and asymptotic normality are proved as the time-scale separation increases and the perturbation intensity vanishes, while the process is observed at only finitely many time points. In particular, for any \(\theta_0\in\Theta\), any \(H\in(1/2,1)\), and any \(\xi>0\),
\begin{equation}
\lim_{(\varepsilon,\eta)\to(0,0)}
\mathbb{P}\left(
\left|
\hat{\theta}_{\mathrm{TFE}}(x_{\mathcal{D}_{\delta,T}}^{\eta,\varepsilon,\theta})
-\theta_0
\right|>\xi
\right)
=
0.
\end{equation}
Similarly, for any fixed set of discrete-time observations, the fluctuations of the estimator are characterised by
\begin{equation}
\frac{1}{\sqrt{\eta}}
\left(
\hat{\theta}_{\mathrm{TFE}}(x_{\mathcal{D}_{\delta,T}}^{\eta,\varepsilon,\theta})
-\theta_0
\right)
\xrightarrow[(\varepsilon,\eta)\to(0,0)]{\mathcal{L}}
\mathcal{N}(0,M(\theta_0,H,N)),
\end{equation}
where \(M(\theta_0,H,N)\) is an explicit covariance matrix.

Secondly, a high-frequency regime is considered, in which the discrete-time observations are collected on a grid with decreasing mesh size. This case is closer to the standard inferential setting for slow/fast systems, in which the sampling scheme must compensate for the fact that the time scales are never infinitely separated. Let \(\delta=T/N\). Under assumptions similar to those used in the finite-sample case, it is shown that
\begin{equation}
\lim_{(\varepsilon+\eta+\delta)\to 0}
\mathbb{P}\left(
\left|
\hat{\theta}_{\mathrm{TFE}}(x_{\mathcal{D}_{\delta,T}}^{\eta,\varepsilon,\theta})
-\theta_0
\right|>\xi
\right)
=
0,
\end{equation}
together with an asymptotic normality result. Although most of the assumptions are technical and similar to those of the finite-sample setting, these results depend crucially on an observation-rate condition of the form
\begin{equation}
\frac{\sqrt{\varepsilon}+\sqrt{\eta}}{\delta}\to 0,
\end{equation}
which is a subsampling condition consistent with the usual inference procedures for slow/fast systems. It is, however, generally more restrictive with respect to the time-scale separation than the usual requirement \(\varepsilon/\delta\to 0\).